\documentclass[11pt]{amsart}

\swapnumbers

\usepackage{amssymb,amsfonts}
\usepackage[mathscr]{eucal}
\usepackage[alphabetic]{amsrefs}

\usepackage[ps,matrix,arrow,curve,cmtip]{xy}

\title{A cartesian presentation of weak $n$-categories}
\author{Charles Rezk}
\date{ \today}
\address{Department of Mathematics \\
University of Illinois at Urbana-Champaign \\ 
Urbana, IL}
\email{rezk@math.uiuc.edu}
\thanks{The author was supported under NSF grant DMS-0505056.}

\numberwithin{equation}{section}

\makeatletter
  \let\c@subsection\c@equation
\makeatother

\theoremstyle{plain}   

\newtheorem{thm}[subsection]{Theorem}
\newtheorem{prop}[subsection]{Proposition}
\newtheorem{cor}[subsection]{Corollary}
\newtheorem{lemma}[subsection]{Lemma}

\newtheorem{conj}[subsection]{Conjecture}

\theoremstyle{remark}
\newtheorem{rem}[subsection]{Remark}

\theoremstyle{plain}

\DeclareMathOperator{\id}{id}
\DeclareMathOperator{\colim}{colim}

\newcommand{\op}{{\operatorname{op}}}
\newcommand{\ob}{{\operatorname{ob}}}

\newcommand{\ra}{\rightarrow}

\newcommand{\xra}{\xrightarrow}
\newcommand{\la}{\leftarrow}

\newcommand{\xla}{\xleftarrow}

\DeclareMathOperator{\hocolim}{hocolim}
\DeclareMathOperator{\holim}{holim}

\newcommand{\set}[2]{{\{\,#1\mid#2\,\}}}

\newcommand{\Z}{\mathbb{Z}}

\newcommand{\dfn}{\textbf}

\def\noloc{\;{:}\,}

\def\defeq{\overset{\mathrm{def}}=}

\newcommand{\forcepar}{\mbox{}\par}

\begin{document}

\newcommand{\Ab}{\mathrm{Ab}} 
\newcommand{\Set}{\mathrm{Set}} 
\newcommand{\sSet}{s\mathrm{Set}} 
\newcommand{\Sp}{\mathrm{Sp}} 
\newcommand{\Psh}{\operatorname{Psh}}
\newcommand{\Mod}{\operatorname{Mod}}
\newcommand{\Bun}{\mathrm{Bun}}
\newcommand{\enr}{\mathrm{enr}}

\newcommand{\Map}{\operatorname{Map}} 
\newcommand{\hMap}{h\mathrm{Map}} 

\newcommand{\RAdj}{\operatorname{RAdj}} 
\newcommand{\CAT}{\mathrm{CAT}} 
\newcommand{\Cat}{\mathrm{Cat}} 
\newcommand{\enrcat}[1]{{#1}\mathrm{\mbox{-}Cat}}
\newcommand{\stcat}[1]{\mathbf{St\mbox{-}}{#1}\mathbf{\mbox{-}Cat}}

\newcommand{\dnerve}{\operatorname{dnerve}} 

\newcommand{\thetaspk}[2]{\Theta_{#1}\mathrm{Sp}_{#2}}
\newcommand{\thetaspkfib}[2]{\Theta_{#1}\mathrm{Sp}^{\mathrm{fib}}_{#2}}
\newcommand{\thetaspenr}[1]{#1\mathrm{\mbox{-}}\Theta\mathrm{Sp}}
\newcommand{\thetagpdk}[2]{\Theta_{#1}\mathrm{Gpd}_{#2}}
\newcommand{\thetagpdenr}[1]{#1\mathrm{\mbox{-}}\Theta\mathrm{Gpd}}

\newcommand{\Se}{\mathrm{Se}} 
\newcommand{\se}{\mathrm{se}} 
\newcommand{\Cpt}{\mathrm{Cpt}} 
\newcommand{\Gpd}{\mathrm{Gpd}} 
\newcommand{\hequiv}{\mathrm{equiv}}

\newcommand{\proj}{\mathrm{proj}} 
\newcommand{\inj}{\mathrm{inj}} 

\newcommand{\sPsh}{s\mathrm{PSh}} 

\newcommand{\timesover}[1]{\underset{#1}{\times}} 

\newcommand{\umap}{\underline{\mathrm{map}}}

\newcommand{\pQ}{\mathcal{Q}}

\begin{abstract}
We propose a notion of weak $(n+k,n)$-category, which we call
$(n+k,n)$-$\Theta$-spaces.  The $(n+k,n)$-$\Theta$-spaces are precisely
the fibrant objects of a certain model category structure on the
category of presheaves of simplicial sets on Joyal's category
$\Theta_n$.  This notion is a generalization of that of complete Segal
spaces (which are precisely the $(\infty,1)$-$\Theta$-spaces).  Our
main result is that the above model category is cartesian.
\emph{Note:} This version of the article has been revised to
correct errors which appear in the published version.
\end{abstract}

\maketitle


\section{Introduction}

In this note, we propose a definition of weak $n$-category, and more
generally, weak $(n+k,n)$-category for all $0\leq n<\infty$ and
$-2\leq k\leq \infty$, called \dfn{$(n+k,n)$-$\Theta$-spaces}.  The
collection of $(n+k,n)$-$\Theta$-spaces forms a category
$\thetaspkfib{n}{k}$, and there is 
a notion for a morphism  in this category to be an equivalence.
The category $\thetaspkfib{n}{k}$ together with the given class of
equivalences has the following desirable property: 
it is cartesian closed, in a way compatible with the equivalences.
More precisely, we have the following.
\begin{enumerate}
\item The category $\thetaspkfib{n}{k}$ is cartesian closed; i.e., it has
  products $Y\times Z$ and 
  function objects $Z^Y$ for any pair 
  of objects $Y,Z$ in $\thetaspkfib{n}{k}$, so that
  $\thetaspkfib{n}{k}(X\times Y,Z) \approx \thetaspkfib{n}{k}(X,Z^Y)$. 
\item If $f\colon X\ra Y$ is an equivalence in $\thetaspkfib{n}{k}$, then
  so are $f\times Z\colon X\times Z\ra Y\times Z$ and  $Z^f\colon
  Z^Y\ra Z^X$.
\end{enumerate}

The  category $\thetaspkfib{n}{k}$ will be defined as the full
subcategory of fibrant objects in a Quillen model category
$\thetaspk{n}{k}$. 
The underlying category of $\thetaspk{n}{k}$ is the category
$\sPsh(\Theta_n)$ of simplicial presheaves on a certain category
$\Theta_n$.  We
equip this category with a  model category structure, 
obtained as the Bousfield localization of the injective model
structure on presheaves with respect to a certain set of morphisms
$\mathscr{T}_{n,k}$.  We 
will prove that $\thetaspk{n}{k}$ is a
\emph{cartesian model category}, i.e., the model category structure is nicely
compatible with the internal function objects of $\sPsh(\Theta_n)$.  
Then $\thetaspkfib{n}{k}$ is the full subcategory of
\emph{fibrant} objects in
$\thetaspk{n}{k}$; equivalences in
$\thetaspkfib{n}{k}$ are just levelwise weak equivalences of
presheaves. 

For $n=0$, the category $\Theta_0$ is the terminal category, so that
$\sPsh(\Theta_0)$ is the category of simplicial sets $\Sp$.  An
$(\infty,0)$-$\Theta$-space is precisely a Kan complex, and a
$(k,0)$-$\Theta$-space is precisely a $k$-truncated Kan complex, i.e.,
a Kan complex with homotopy groups vanishing above dimension $k$.

For $n=1$, the category $\Theta_1$ is the category $\Delta$ of finite
ordinals, so that $\sPsh(\Theta_1)$ is the category of simplicial
spaces.  An $(\infty,1)$-$\Theta$-space is precisely a \emph{complete
  Segal space}, in the sense of \cite{rezk-ho-theory-of-ho-theory}.

The category $\Theta_n$ which we use was introduced by Joyal
\cite{joyal-theta-note}, as part of an attempt to define a notion of
weak $n$-category generalizing the notion of quasicategory .  Sketches
of these ideas can be found in \cite{leinster-survey} and
\cite{cheng-lauda-illustrated-guide}.   The category $\Theta_n$ has
also been studied by Berger \citelist{\cite{berger-cellular-nerve}
  \cite{berger-iterated-wreath}}, with particular application to the
theory of iterated loop spaces.

\subsection{The categories $\Theta_n$}

We will give an informal description of Joyal's categories $\Theta_n$ 
here, suitable for our purposes; our description is essentially the
same as that given in \cite{berger-iterated-wreath}*{\S3}.  It is most
useful for us 
to regard 
$\Theta_n$ as a full subcategory of $\stcat{n}$, the category of
strict $n$ categories.  Thus, $\Theta_0$ is the full subcategory of
$\stcat{0}=\Set$ consisting of the terminal object.  The category $\Theta_1$
is the full subcategory of $\stcat{1}$ consisting of object $[n]$ for
$n\geq0$, where $[n]$ represents the free strict $1$-category on the
diagram
\[
\xymatrix{
{0} \ar[r]
& {1} \ar[r]
& {\cdots} \ar[r]
& {(n-1)} \ar[r]
& {n} 
}\]
Thus, $\Theta_1=\Delta$, the usual simplicial indexing category.  The
category 
$\Theta_2$ is the full subcategory of $\stcat{2}$ consisting of
objects which are denoted $[m]([n_1],\dots,[n_m])$ for
$m,n_1,\dots,n_m\geq0$.  This represents the strict $2$-category $C$ which
is ``freely generated'' by: 
objects $\{0,1,\dots,m\}$, and morphism categories $C(i-1,i)=[n_i]$.  
For
instance, the object $[4]([2],[3],[0],[1])$ in $\Theta_2$ corresponds to
the ``free $2$-category'' on the following picture:
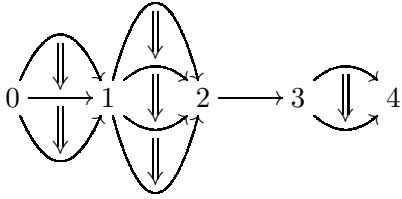
\[\xymatrix{
{0} \ar@/^2pc/[r]^{}="10"  \ar[r]^{}="11" \ar@/_2pc/[r]^{}="12"
\ar@{=>}"10";"11" \ar@{=>}"11";"12"
& {1} \ar@/^3pc/[r]^{}="20" \ar@/^1pc/[r]^{}="21"
\ar@/_1pc/[r]^{}="22" \ar@/_3pc/[r]^{}="23" \ar@{=>}"20";"21"
\ar@{=>}"21";"22" \ar@{=>}"22";"23"  
& {2} \ar[r]
& {3} \ar@/^1pc/[r]^{}="30" \ar@/_1pc/[r]^{}="31" \ar@{=>}"30";"31"
& {4}
}\]
In general, the objects of $\Theta_n$ are of the form
$[m](\theta_1,\dots,\theta_m)$, where $m\geq0$ and the $\theta_i$ are
objects of $\Theta_{n-1}$; this object corresponds to the strict
$n$-category $C$ 
``freely generated'' by: objects $\{0,\dots,m\}$, and a strict
$(n-1)$-category of morphisms $C(i-1,i)=\theta_i$.  The morphisms of
$\Theta_n$ are functors between strict $n$-categories.

We make special note of objects $O_0,\dots,O_n$ in $\Theta_n$.  These
are defined recursively by: $O_0=[0]$, and $O_k=[1](O_{k-1})$ for
  $k=1,\dots,n$.  Thus, the object $O_k$ in $\Theta_n$ corresponds to the
  ``freestanding $k$-cell'' in $\stcat{n}$.

\subsection{Informal description of $\Theta$-spaces}
\label{subsec:informal-description}

We will start by describing $\thetaspkfib{n}{\infty}$, the category of
$(\infty,n)$-$\Theta$-spaces.  Let $\Sp$ denote the category of
simplicial sets.  We will regard objects of $\Sp$ as ``spaces''; the
following definitions are perfectly sensible if objects of $\Sp$ are
taken to be actual topological spaces (compactly generated).

An object of $\thetaspkfib{n}{\infty}$ is a simplicial presheaf on
$\Theta_n$ (i.e., a
functor $X\colon \Theta_n^\op\ra \Sp$), satisfying three conditions:
\begin{enumerate}
\item [(i)] an \emph{injective fibrancy} condition,
\item [(ii)] a \emph{Segal} condition, and
\item [(iii)] a \emph{completeness} condition.
\end{enumerate}
A morphism $f\colon X\ra Y$ of $\thetaspkfib{n}{\infty}$ is a morphism
of simplicial 
presheaves; the morphism $f$ is said to be an \dfn{equivalence} (or
\dfn{weak equivalence}) if it is a ``levelwise'' weak equivalence of
simplicial presheaves, i.e., if $f(\theta)\colon X\theta\ra Y\theta$
is a weak equivalence of simplicial sets for all $\theta\in \ob
\Theta$. 

The injective fibrancy condition (i) says that $X$ has a right lifting
property with respect to maps in $\sPsh(\Theta_n)$ which are both
monomorphisms and levelwise weak equivalences; that is, $X$ is fibrant in the
\emph{injective} model structure on $\sPsh(\Theta_n)$.

The Segal condition (ii) says that for all objects $\theta$ of
$\Theta_n$, the space $X(\theta)$ is weakly equivalent to an
inverse limit of a certain diagram of spaces $X(O_k)$; taken together
with the injective 
fibrancy condition, this inverse limit is in fact the
homotopy inverse limit of the given diagram.  For $n=1$, the Segal
condition amounts to requiring that the ``Segal map''
\[
X([m])\ra X([1])\timesover{X([0])} \cdots\timesover{X([0])} X([1]) \approx 
X(O_1)\timesover{X(O_0)}\cdots \timesover{X(O_0)} X(O_1)
\]
be a weak equivalence for all $m\geq2$.  As an example of how the
Segal condition works 
for $n=2$, the space $X([4]([2],[3],[0],[1]))$ is required to be weakly
equivalent to
\[
\left(X(O_2)\timesover{X(O_1)}X(O_2)\right)\timesover{X(O_0)}
\left(X(O_2)\timesover{X(O_1)}X(O_2)\timesover{X(O_1)}
X(O_2)\right)\timesover{X(O_0)} 
X(O_1) \timesover{X(O_0)} X(O_2).
\]

The completeness condition (iii) says that the space $X(O_k)$ should
behave like the ``moduli space'' of $k$-cells in a
$(\infty,n)$-category.  That is, if the points of $X(O_k)$ correspond
to individual $k$-cells, such points should be connected by a path in
$X(O_k)$ if they represent ``equivalent'' $k$-cells, there should be a
homotopy between paths for every ``equivalence between equivalences'',
and so on.  It turns 
out that the way to enforce this is to require that, for
$k=1,\dots,n$,  the map
$X(i_k)\colon X(O_{k-1})\ra X(O_{k})$ which encodes ``send a $(k-1)$-cell  to
its identity $k$-morphism'' should induce a weak equivalence of spaces
\[
X(O_{k-1}) \ra X(O_{k})^\hequiv.
\]
Here $X(O_{k})^\hequiv$ is the union of those path components of
$X(O_{k})$  which consist of $k$-morphisms which are
``$k$-equivalences''. 
Thus, the completeness condition asserts that the moduli space of $(k-1)$-cells
is weakly equivalent to the moduli space of
$k$-equivalences.   

The category $\thetaspk{n}{k}$ of $(n+k,n)\mbox{-}\Theta$-spaces is
obtained by imposing an additional
\begin{enumerate}
\item [(iv)] \emph{$k$-truncation} condition.
\end{enumerate}
To state this, we need the moduli space
$X(\partial O_m)$ of ``parallel pairs of $(m-1)$-morphisms'' in $X$.  
This space is defined inductively as an inverse limit of the spaces
$X(O_m)$, so that
\[
X(\partial O_m)\defeq \lim \bigl(X(O_{m-1})\ra X(\partial O_{m-1}) \la
X(O_{m-1})\bigr),
\]
with $X(\partial O_0)=1$.  Then the $k$ truncation condition asserts
that the fibers of $X(O_n)\ra X(\partial O_n)$ are
$k$-types, i.e., have vanishing homotopy groups in all dimensions
greater than $k$.

The above definition is examined in detail in \S\ref{sec:theta-spaces}.

\subsection{Presentations and enriched model categories}

Our construction of a cartesian model category structure is a special
case of a general procedure, which associates to certain kinds of model
categories $M$ a new model category $\thetaspenr{M}$; we may regard  this
as being analogous to the procedure which associates to a category $V$
with finite products the category $\enrcat{V}$ of categories enriched
over $V$.

Specifically,
suppose we are given a pair $(C,\mathscr{S})$ consisting
of a small category $C$ and a set $\mathscr{S}$ of morphisms in
$\sPsh(C)$;
this data is called a \dfn{presentation}, following the
treatment of \cite{dugger-universal-homotopy-theories}.   (Here,
$\sPsh(C)$ denotes the category of presheaves of simplicial sets on
$C$.) 
Let
$M=\sPsh(C)^\inj_{\mathscr{S}}$ denote the model 
category structure on $\sPsh(C)$ obtained by Bousfield localization of
the injective model structure with respect to $\mathscr{S}$.  We
define a new presentation $(\Theta C,
\mathscr{S}_\Theta)$, and thus obtain a model category
$\thetaspenr{M}\defeq \sPsh(\Theta
C)^\inj_{\mathscr{S}_\Theta}$.  The category $\Theta C$ is a ``wreath
product'' of $\Delta$ with $C$, as defined by
\cite{berger-iterated-wreath} (see \S\ref{sec:theta-construction}),
while the set $\mathscr{S}_\Theta$ 
consists of some maps 
built from elements of $\mathscr{S}$, together with certain
``Segal'' and ``completeness'' maps (this set is described in
\S\ref{sec:presentation-theta-css}). 

Our main result is the following
theorem.
\begin{thm}\label{thm:main-intro}
Let $M=\sPsh(C)^\inj_{\mathscr{S}}$ for some presentation
$(C,\mathscr{S})$.  If $M$ is a cartesian 
model category, then $\thetaspenr{M}$ is also a cartesian model
category.
\end{thm}
This theorem is a straightforward generalization of the main theorem of
\cite{rezk-ho-theory-of-ho-theory}, which proves the theorem for the
special case $(C,\mathscr{S})=(1,\varnothing)$ (in which case $M=\Sp$,
and thus $\thetaspenr{M}$ is the category of simplicial spaces with the
complete Segal space model structure.)

The model categories for $(n+k,n)$-$\Theta$-spaces are obtained
iteratively, so that
\[
\thetaspk{n+1}{k} \defeq \thetaspenr{(\thetaspk{n}{k})},
\]
starting with $\thetaspk{0}{k}=\Sp_k$, where $\Sp_k$ is the Bousfield
localization of $\Sp$ whose fibrant objects are Kan complexes which
are $k$-types.  Applying the theorem inductively shows that
$\thetaspk{n}{k}$ are cartesian model categories.  The category
$\thetaspkfib{n}{k}$ is defined to be the full subcategory of fibrant
objects in the model category $\thetaspk{n}{k}$.

\subsection{The relationship between $\thetaspenr{M}$ and $\enrcat{M}$}
\label{subsec:conjectural-relationship}

If $M$ is a cartesian model category, then we may certainly consider
$\enrcat{M}$, the category of small categories enriched over $M$.
Given an object $X$ of $\enrcat{M}$, let $hX$ denote the ordinary
category whose objects are the same as $X$, and whose morphism sets
are given by $hX(a,b) \defeq 
hM(1,X(a,b))$, where $hM$ denotes the homotopy category of $M$. 
Let
us say that a morphism $f\colon X\ra Y$ of objects of $\enrcat{M}$ is
a \dfn{weak equivalence} if 
\begin{enumerate}
\item  for each pair of objects $a,b$ of $X$, the induced map
  $X(a,b)\ra Y(fa,fb)$ is a weak equivalence in $M$, and
\item  the induced functor $hX\ra hY$ is an equivalence of
  $1$-categories.
\end{enumerate}
We can make the  following conjecture.
\begin{conj}
Let $M=\sPsh(C)^\inj_{\mathscr{S}}$ for some presentation
$(C,\mathscr{S})$, and suppose that $M$ is a cartesian model category.
Then there is a model
category structure on $\enrcat{M}$ with the above weak equivalences, 
and a Quillen equivalence
\[
\enrcat{M} \approx \thetaspenr{M}.
\]
\end{conj}
For the case of $M=\Sp$, the conjecture follows from 
theorems of Bergner \citelist{\cite{bergner-model-cat-simplicial-cat},
  \cite{bergner-three-models}}.

\subsection{Why is this a good notion of weak $n$-category?}

We propose that $(n+k,n)\mbox{-}\Theta$-spaces are a model for
weak $(n+k,n)$-categories.  Some points in its favor are the
following.
\begin{enumerate}
\item Our notion of $(\infty,1)\mbox{-}\Theta$-spaces is precisely
  what we called a \emph{complete Segal space} in
  \cite{rezk-ho-theory-of-ho-theory}.  This is recognized as a
  suitable model for 
  $(\infty,1)$-categories, due to work of Bergner
  \cite{bergner-three-models} and
  Joyal-Tierney \cite{joyal-tierney-quasicats-segal-spaces}.
\item  As noted above \eqref{subsec:conjectural-relationship}, the
  definition of $(n+k,n)$-$\Theta$-spaces is 
  a special case of a more general construction, which conjecturally
  models ``homotopy theories enriched over a cartesian model
  category''.   In particular, a consequence of our conjecture would
  be a Quillen equivalence 
\[
\enrcat{(\thetaspk{n}{k})} \approx \thetaspenr{(\thetaspk{n}{k})} =
\thetaspk{n+1}{k}.
\]
That is, $(n+1+k,n+1)$-$\Theta$-spaces are (conjecturally) ``the
same'' as categories enriched over $(n+k,n)$-$\Theta$-spaces.

\item Our notion satisfies the ``homotopy hypothesis''.  There is an
  evident notion of groupoid object in $\thetaspk{n}{k}$, and one can
  show that the full subcategory of such groupoid objects models
  $\Sp_{n+k}$, the homotopy theory of $(n+k)$-types
  \eqref{subsec:homotopy-hypothesis}. 

\item More generally, it is understood that $n$-tuply monoidal
  $k$-groupoids should   correspond to ``$k$-types of
  $E_n$-spaces'', where $E_n$ is a version of the little $n$-cubes
  operad; furthermore, $n$-tuply \emph{groupal} $k$-groupoids should
  correspond to ``$k$-types of $n$-fold loop spaces'' (see, for
  instance, \cite{baez-dolan-categorification}*{\S3}).  In terms of
  our models, $n$-tuply groupal $k$-groupoids are objects
  $X$ of $\thetaspk{n}{k}$ for which (i)
  $X(O_j)\approx 1$ for $j<n$, and (ii) $X(O_n)^\hequiv\approx
  X(O_n)$, and one would conjecture that the full subcategory of such
  objects in $\thetaspk{n}{k}$ should model $k$-types of $n$-fold
  loop spaces.  That this is in fact the case is apparent from the
  results of Berger 
  \cite{berger-iterated-wreath}. 
\end{enumerate}

As noted above, the theory of $\Theta$-spaces is consciously a
generalization of the theory of complete Segal spaces, which is one of
a family of models for $(\infty,1)$-categories based on simplicial
objects.  
A reasonable approach to producing a generalization of these
ideas is to use multi-simplicial objects; proposals for this include
Tamsamani's theory of weak $n$-categories \cite{tamsamani-n-cat}, the
Segal $n$-categories of 
Hirschowitz and Simpson
\cite{hirschowitz-simpson-descente-pour-n-champs}, and more recent work
by Barwick and Lurie on multi-simplicial generalizations of complete Segal
spaces (see for instance \cite{barwick-extended-research-statement}
and \cite{lurie-classification-tft}).  Although all these
constructions appear to give good models for $(\infty,n)$-categories,
it is not clear to me  that any of them result in a Quillen model
category which satisfies all of the following:
(i) it models the homotopy theory of $(\infty,n)$-categories with
the correct notion of equivalence, (ii) it is a cartesian model
category, and (iii) it is a simplicial model category.    It does
appear that the Hirschowitz-Simpson model satisfies (i) and (ii), but
it does not satisfy (iii).  The multi-simplicial complete Segal space
model of Barwick and Lurie does satisfy (i) and (iii), but does not
appear to satisfy (ii) (when $n>1$.)

\subsection{Corrections}
\label{subsec:corrections}

The published version of this paper \cite{rezk-theta-n-spaces},
contains a false statement 
\eqref{prop:hocolim-poset}, and an incorrect proof of
\eqref{prop:product-of-covers} based on that false statement.  A
correction was published in \cite{rezk-theta-n-spaces-correction}.   I
have incorporated the corrections into the body of this online
version; the numbering of theorems has not changed from the
published version.

\section{Cartesian model categories and cartesian presentations}

\subsection{Cartesian closed categories}

A category $V$ is said to be \dfn{cartesian closed} if it
has finite products, and if for all $X,Y\in \ob V$ there is an
\dfn{internal function object} $Y^X$, which comes equipped with a
natural isomorphism
\[
V(T, Y^X)\approx V(T\times X, Y).
\]
Examples of cartesian categories include the category of sets, and the
category of presheaves of sets on a small category $C$.

We will write $\varnothing$ for some chosen initial object in a
cartesian closed category $V$.

\subsection{Cartesian model categories}
\label{subsec:cartesian-model-categories}

We will say that a Quillen model category $M$ is \dfn{cartesian} if it is
cartesian closed as a category, if the terminal object is cofibrant, 
and if the following equivalent
properties hold.
\begin{enumerate}
\item If $f\colon A\ra A'$ and $g\colon B\ra B'$ are cofibrations in
  $M$, then the induced map $h\colon A\times B'\amalg_{A\times B} A'\times B\ra
  A'\times B'$ is a cofibration; if in addition either $f$ or $g$ is a
  weak equivalence then so is $h$.
\item If $f\colon A\ra A'$ is a cofibration and $p\colon X\ra X'$ is a
  fibration in $M$, then the induced map $q\colon (X')^{A'}\ra
  (X')^A\times_{X^A} X^{A'}$ is a fibration; if in addition either $f$ or
  $p$ is a weak equivalence then so is $q$.
\end{enumerate}

(This notion is a bit stronger than Hovey's notion of symmetric
monoidal model category as applied to cartesian closed categories
\cite{hovey-model-categories}*{4.2}; he does not require the unit
object to be cofibrant, but rather imposes a weaker condition.)  

\subsection{Spaces}

Let $\Sp$ denote the category of simplicial sets, equipped with the
standard Quillen model structure.  We call this the model category of
\dfn{spaces}.  It is standard that $\Sp$ is a cartesian model
category.

We will often use topological flavored language when discussing
objects of $\Sp$, even though such are objects are not topological
spaces but simplicial sets.  Thus, a ``point'' in a ``space'' is
really a $0$-simplex of a simplicial set, a ``path'' is a $1$-simplex,
and so on.

\subsection{Simplicial presheaves}

Let $C$ be a small category, and let $\sPsh(C)$ denote the category of
\dfn{simplicial presheaves} on $C$, i.e., the category of
contravariant functors $C^\op\ra \Sp$.  

A simplicial presheaf $X$ is said to be \dfn{discrete} if each $X(c)$
is a discrete simplicial set; the full subcategory of discrete objects
in $\sPsh(C)$ is equivalent to the category of presheaves of
\emph{sets} on $C$, and we will implicitly identity the two.

Let $F_C\colon C\ra \sPsh(C)$ denote the \dfn{Yoneda functor}; thus $F_C$
sends an object $c\in \ob C$ to the presheaf $F_Cc=C(\mbox{-},c)$.  Observe
that the presheaf $F_Cc$ is discrete.  When the context is understood
we may write $F$ for $F_C$.

Let $\Gamma_C\colon \sPsh(C)\ra \Sp$ denote the \dfn{global sections
  functor}, which sends a functor $X\colon C^\op\ra \Sp$ to its
limit.  The functor $\Gamma$ is right adjoint to the functor $\Sp\ra
\sPsh(C)$ which sends a simplicial set $K$ to the constant presheaf 
with value $K$ at each object of $C$.
Note that if $C$ has a terminal object $[0]$, then $\Gamma
X\approx X([0])$.  
When the context is understood we may write $\Gamma X$ for $\Gamma_CX$.

For $X,Y$ in $\sPsh(C)$, we write $\Map_C(X,Y)\defeq \Gamma(Y^X)$; this
is called the \dfn{mapping space}.  Thus, $\sPsh(C)$ is enriched over $\Sp$.
Note that if $c\in \ob C$, then we have
\[
X(c) \approx \Gamma(X^{F(c)}) \approx \Map_C(F(c),X).
\]
When the context is understood we may write $\Map(X,Y)$ for
$\Map_C(X,Y)$.

\subsection{Model categories for simplicial presheaves}

Say that a map $f\colon X\ra Y\in \sPsh(C)$ is a \dfn{levelwise weak
  equivalence} if each map $f(c)\colon X(c)\ra Y(c)$ is a weak
equivalence in $\Sp$ for all $c\in \ob C$.
There are two standard model
category structures we can put on $\sPsh(C)$ with these weak
equivalences, called the projective and
injective structures; they are Quillen equivalent to each other.

The \dfn{projective} structure is characterized
by requiring that $f\colon X\ra Y\in \sPsh(C)$ be a fibration if and
only if $f(c)\colon X(c)\ra Y(c)$ is one in $\Sp$ 
for all $c\in \ob C$.  We write $\sPsh(C)^\proj$ for the category of
presheaves of simplicial sets on $C$ equipped with the projective
model structure.  

The \dfn{injective} structure is characterized by
requiring that $f\colon X\ra Y\in \sPsh(C)$ is a cofibration if and
only if $f(c)\colon X(c)\ra Y(c)$ is one in $\Sp$ 
for all $c\in \ob C$.  We write $\sPsh(C)^\inj$ for the category of
presheaves of simplicial sets on $C$ equipped with the injective model
structure.  

The identity functor provides a Quillen equivalence $\sPsh(C)^\proj
\rightleftarrows \sPsh(C)^\inj$.

Both the projective and injective model category structures are
cofibrantly generated and proper, and have functorial factorizations.
They are also both simplicial model categories.

Given object $X,Y$ in $\sPsh(C)$, we write $\hMap_C(X,Y)$ for the
\dfn{derived mapping space} of maps from $X$ to $Y$.  This is a
homotopy type in $\Sp$, defined so that for any cofibrant
approximation $X^c\ra X$ and fibrant approximation $Y\ra Y^f$, the
derived mapping space $\hMap_C(X,Y)$ is weakly equivalent to the space
of maps $\Map_C(X^c,Y^f)$.  Note that in the above, we may pick either
the injective or projective model category structures in order to make
our replacements.

\subsection{The injective model structure}

The injective model structure has a few additional properties which
are of importance to us.  In particular, 
\begin{enumerate}
\item every object of $\sPsh(C)^\inj$ is
cofibrant, and
\item every \emph{discrete} object of $\sPsh(C)^\inj$ is fibrant.
\end{enumerate}

Furthermore, we have the following.
\begin{prop}
The model category $\sPsh(C)^\inj$ is a cartesian model category.
\end{prop}
\begin{proof}
This is immediate from characterization (1) of cartesian model category.
\end{proof}

\subsection{Presentations}

A \dfn{presentation} is a pair $(C,\mathcal{S})$ consisting of a small
category $C$ and a set $\mathscr{S}=\{s\colon S\ra S'\}$ of morphisms
in $\sPsh(C)$.  We say that an object $X$ of $\sPsh(C)$ is
\dfn{$\mathscr{S}$-local} if for all morphisms $s\colon S\ra S'$ in
$\mathscr{S}$, the induced map
\[
\hMap(s,X)\colon \hMap(S',X) \ra \hMap(S,X)
\]
is a weak equivalence of spaces.  We say that a morphism $f\colon A\ra
B$ in $\sPsh(C)$ is an \dfn{$\mathscr{S}$-equivalence} if the induced
map
\[
\hMap(f,X)\colon \hMap(B,X)\ra \hMap(A,X)
\]
is a weak equivalence of spaces for all $\mathscr{S}$-local objects
$X$.  The collection 
of $\mathscr{S}$-equivalences is denoted $\overline{\mathscr{S}}$; we
have that $\mathscr{S}\subset \overline{\mathscr{S}}$. 

Let $(C,\mathscr{S})$ be a presentation, let $X$ be an object of
$\sPsh(C)$, and let $X\ra X^f$ denote a fibrant replacement of $X$ in the
\emph{injective} model structure.  Since every object is cofibrant in
the injective model structure, we have that $X$ is $\mathscr{S}$-local
if and only if $\Map(S',X^f)\ra \Map(S,X^f)$ is a weak equivalence for
all $s\in \mathscr{S}$.

\subsection{Cartesian presentations}

Let $(C,\mathscr{S})$ be a presentation.
Given an object in $X$ of $\sPsh(C)$, we say it is
\dfn{$\mathscr{S}$-cartesian local} if for all $s\colon S\ra S'$ in
$\mathscr{S}$, the induced map
\[
Y^s\colon Y^{S'}\ra Y^{S}
\]
is a levelwise weak equivalence, where $X\ra Y$ is some choice of
fibrant replacement in $\sPsh(C)^\inj$.

\begin{prop}\label{prop:cartesian-local-object-charac}
Let $X$ be an object of $\sPsh(C)$, and choose some fibrant
replacement $X\ra Y$ in $\sPsh(C)^\inj$.  Then $X$ is
$\mathscr{S}$-cartesian local if and only if for all $c\in \ob C$, the
function object $Y^{F(c)}$ is $\mathscr{S}$-local.
\end{prop}
\begin{proof}
Immediate from the isomorphism
\[
Y^S(c) \approx \Map(F(c),Y^S)\approx \Map(S,Y^{F(c)}).
\]
\end{proof}

Observe that every $\mathscr{S}$-cartesian local object is necessarily
$\mathscr{S}$-local, since $\Map(S,Y)\approx \Map(1,Y^S)$; however,
the converse need not hold. 
We say that a presentation $(C,\mathscr{S})$ is a \dfn{cartesian
  presentation} if every $\mathscr{S}$-local object is
$\mathscr{S}$-cartesian local.

\begin{prop}\label{prop:cart-pres-equiv-charac}
Let $(C,\mathscr{S})$ be a presentation.  The following are
equivalent.
\begin{enumerate}
\item $(C,\mathscr{S})$ is a cartesian presentation.
\item For all $\mathscr{S}$-fibrant $X$ in $\sPsh(C)$ and all $c\in \ob
  C$, the object $X^{F(c)}$ is $\mathscr{S}$-local.
\item For all $s\colon S\ra S'\in \mathscr{S}$ and all $c\in \ob C$,
  the map $s\times \id \colon S\times Fc \ra S'\times Fc$ is in
  $\overline{\mathscr{S}}$. 
\end{enumerate}
\end{prop}
\begin{proof}
Immediate from \eqref{prop:cartesian-local-object-charac}.
\end{proof}

\begin{prop}
If $(C,\mathscr{S})$ is a cartesian presentation, then $f,g\in
\overline{\mathscr{S}}$ imply $f\times g\in \overline{\mathscr{S}}$.  
\end{prop}

\subsection{Localization}

Given a presentation $(C,\mathscr{S})$, we write
$\sPsh(C)^\proj_{\mathscr{S}}$ and 
$\sPsh(C)^\inj_{\mathscr{S}}$ for the model category structures on
$\sPsh(C)$ obtained by Bousfield localization of the
original projective and injective model structures on $\sPsh(C)$.
These model categories are again Quillen equivalent to each other.
We
will set out the details in the case of the injective model structure.
\begin{prop}
Given a presentation $(C,\mathscr{S})$ there exists a cofibrantly
generated, left proper, simplicial model category structure
on
$\sPsh(C)$ which is characterized by the following properties.
\begin{enumerate}
\item [(1)] The cofibrations are exactly the monomorphisms.
\item [(2)] The fibrant objects are precisely the injective fibrant
  objects which are $\mathscr{S}$-local.  (We call these the
  \dfn{$\mathscr{S}$-fibrant} objects.)
\item [(3)] The weak equivalences are precisely the
  $\mathscr{S}$-equivalences. 
\end{enumerate}
Furthermore, we have that
\begin{enumerate}
\item [(4)] A levelwise weak equivalence $g\colon X\ra Y$ is an
  $\mathscr{S}$-equivalence, and the converse holds if both $X$ and $Y$ are
  $\mathscr{S}$-local.
\item [(5)] An $\mathscr{S}$-fibration $g\colon X\ra Y$ is an
  injective fibration, and the converse holds if both $X$ and $Y$ are
  $\mathscr{S}$-fibrant. 
\end{enumerate}
\end{prop}
\begin{proof}
This is an example of \cite{hirschhorn-localization}*{Thm.\ 4.1.1},
since $\sPsh(C)^\inj$ is a left proper cellular model category.
\end{proof}
We will write $\sPsh(C)^\inj_{\mathscr{S}}$ for the above model
structure, which is called the \dfn{$\mathscr{S}$-local injective
  model structure}.

Observe that if $(C,\mathscr{S})$ and $(C,\mathscr{S}')$ are two
presentations on $C$ such that the $\mathscr{S}$-local objects are
precisely the same as the $\mathscr{S}'$-local objects, then
$\sPsh(C)^\inj_{\mathscr{S}}= \sPsh(C)^\inj_{\mathscr{S}'}$. 

\subsection{Quillen pairs between localizations}

\begin{prop}\label{prop:quillen-pair-loc}
Suppose that $(C,\mathscr{S})$ and $(D,\mathscr{T})$ are
presentations, and that we have a Quillen pair $G_\#\colon
\sPsh(C)^\inj\rightleftarrows 
\sPsh(D)^\inj$  (with $G_\#$ the left adjoint).  Then
\[
G_\#\colon \sPsh(C)^\inj_{\mathscr{S}} \rightleftarrows
\sPsh(D)^\inj_{\mathscr{T}}\noloc G^*
\]
is a Quillen pair if and only if either of the two following
equivalent statements hold.
\begin{enumerate}
\item For all $s\in \mathscr{S}$, $G_\#s\in \overline{\mathscr{T}}$.
\item For all $\mathscr{T}$-fibrant objects $Y$ in $\sPsh(D)$, $G^*Y$
  is $\mathscr{S}$-fibrant.
\end{enumerate}
\end{prop}
\begin{proof}
This is 
straightforward from the definitions.
\end{proof}

\subsection{$\mathscr{S}$-equivalences and homotopy colimits}

The following proposition says that the class of
$\mathscr{S}$-equivalences is closed under homotopy colimits.  We
refer to \cite{hirschhorn-localization} for background on homotopy
colimits. 

\begin{prop}\label{prop:saturations-hocolims}
Let $D$ be a small category, and let $(C,\mathscr{S})$ be a
presentation.  Suppose that $\alpha\colon G\ra H$ is a natural
transformation of functors $D\ra \sPsh(C)^\inj$, and consider the
induced map 
\[
\hocolim_D \alpha\colon \hocolim_D G \ra \hocolim_D H
\]
on homotopy colimits, where these homotopy colimits are computed in
the injective model structure on $\sPsh(C)$.  If $\alpha(d)\in
\overline{S}$ for all $d\in \ob D$, then $\hocolim_D \alpha \in \overline{S}$.
\end{prop}
\begin{proof}
In general, the map $\hMap_C(\hocolim_D H, X) \ra \hMap(\hocolim_D G,X)$ is
weakly equivalent to the map $\holim_D \hMap(H,X) \ra \holim_D \hMap(G,X)$;
the result follows by considering the case when $X$ is
$\mathscr{S}$-local. 
\end{proof}

\begin{prop}\label{prop:hocolim-poset}
(As explained in the introduction \S\ref{subsec:corrections}, this
proposition as stated in the published version was incorrect.)  
\end{prop}
%

Finally, we record the following fact, which we use in
\S\ref{sec:segal-objects}.  For a category $C$ and an 
object $A$ in $C$, we write $A\backslash C$ for the slice category of
objects under $A$ in $C$.
\begin{prop}\label{prop:left-quillen-saturations}
Let $C$ be a small category, and let $(D,\mathscr{S})$ be a
presentation.  Suppose that $\alpha\colon G\ra H$ is a natural
transformation of functors $\sPsh(C)\ra \sPsh(D)$.  Suppose the
following hold.
\begin{enumerate}
\item  The functors $X\mapsto (G(\varnothing)\ra G(X))\colon \sPsh(C)^\inj\ra
  G(\varnothing)\backslash \sPsh(D)^\inj$ and $X\mapsto
  (H(\varnothing)\ra H(X))\colon
  \sPsh(C)^\inj\ra H(\varnothing)\backslash \sPsh(D)^\inj$ are left
  Quillen 
  functors. 
\item The map $\alpha(\varnothing)\colon G(\varnothing)\ra
  H(\varnothing)$ is a monomorphism, and is in
  $\overline{\mathscr{S}}$.  
\item The maps $\alpha(Fc)\colon G(Fc)\ra H(Fc)$ are in
  $\overline{\mathscr{S}}$ for all $c\in \ob C$.  
\end{enumerate}
Then $\alpha(X) \in \overline{\mathscr{S}}$ for all objects $X$ of
$\sPsh(C)$. 
\end{prop}
\begin{proof}
In the special case in which $\alpha(\varnothing)\colon
G(\varnothing)\ra H(\varnothing)$ is an isomorphism, note that since
(i) every object of $\sPsh(C)$ is levelwise weakly equivalent to a
homotopy colimit of some diagram of free objects, and (ii) left
Quillen functors preserve homotopy colimits, the result follows using
\eqref{prop:saturations-hocolims}. 

For the general case, factor $\alpha$ into $G\xra{\beta} K\xra{\gamma}
H$ where $K(X) =  
G(X)\cup_{G(\varnothing)} H(\varnothing)$.   The map $\beta(X)$ is a
pushout of the $\mathscr{S}$-local equivalence $\alpha(\varnothing)$
along the map $G(\varnothing)\ra G(X)$ which is an injective
cofibration by (i); thus $\beta(X)\in \overline{\mathscr{S}}$.  
The special case described above applies to show that $\gamma(X)\in
\overline{\mathscr{S}}$.  Thus, the composite $\alpha(X)\in
\overline{\mathscr{S}}$, as desired. 
\end{proof}

\subsection{Cartesian presentations give cartesian model categories}

\begin{prop}
The model category $\sPsh(C)^\inj_{\mathscr{S}}$ is cartesian if and
only if $(C,\mathscr{S})$ is a cartesian presentation.  In particular,
$\sPsh(C)^\inj_{\mathscr{S}}$  is a cartesian model category if for all 
$\mathscr{S}$-fibrant $Y$ and all $c\in \ob C$, the object $Y^{F(c)}$
is $\mathscr{S}$-fibrant.
\end{prop}
\begin{proof}
It is clear that the terminal object is cofibrant in
$\sPsh(C)^\inj_{\mathscr{S}}$, so it suffices to show that
$(C,\mathscr{S})$ is a cartesian presentation if and only if condition
(1) of \S\ref{subsec:cartesian-model-categories} holds.  Let $f\colon
A\ra A'$ and $g\colon B\ra B'$ be cofibrations in
$\sPsh(C)^\inj_{\mathscr{S}}$.  It is clear that the map
\[
h\colon A\times B'\amalg_{A\times B} A'\times B\ra A'\times B'
\]
is a cofibration in any case, so it suffices to show that
``$W$ is $\mathscr{S}$-fibrant implies $W$ is $\mathscr{S}$-cartesian
fibrant'' is equivalent to
``$g\in \overline{\mathscr{S}}$ implies $h\in
\overline{\mathscr{S}}$''.  Since $\sPsh(C)^\inj$ is a cartesian model
category, for a cofibration $g$ as above and all
injective fibrant $W$ we have that $W^g$ is a levelwise weak
equivalence if and only if $\Map(h,W)$ is a  weak equivalence.
The result follows by considering the case of $\mathscr{S}$-fibrant $W$.

\end{proof}

\subsection{$k$-types}

Observe that $\Sp\approx \sPsh(1)^\inj$.  For any integer $k\geq-2$, let
\[
\Sp_k \defeq \sPsh(1)^\inj_{\{\partial \Delta^{k+2}\ra \Delta^{k+2}\}}.
\]
This is called the model category of \dfn{$k$-types}.  The fibrant
objects are precisely the fibrant simplicial sets whose homotopy
groups vanish in dimensions greater than $k$.  This is a cartesian
model category.

\section{The $\Theta$ construction}
\label{sec:theta-construction}

The $\Theta$ construction was introduced by Berger in
\cite{berger-iterated-wreath}, where, with good cause, he calls it the
``categorical wreath product over $\Delta$''; what we are calling
$\Theta C$, he calls $\Delta \wr C$.

\subsection{The category $\Delta$}

We write $\Delta$ for the standard category of finite ordinals; the
objects are $[m]=\{0,1,\dots,m\}$ for $m\geq0$, and the morphisms are
weakly monotone maps.  We will use the following transparent notation
to describe 
particular maps in $\Delta$; we write 
\[
\delta^{k_0k_1\cdots k_m}\colon [m]\ra [n]
\]
for the function defined by $i\mapsto k_i$.

We call a morphism $\delta\colon [m]\ra [n] \in \Delta$ an
\dfn{injection} or \dfn{surjection} if it is so as a map of sets.  We
say that $\delta$ is \dfn{sequential} if 
\[
\delta(i-1)+1 \geq \delta(i)\qquad \text{for all $i=1,\dots,m$.}
\]
Observe that every surjection is sequential.

\subsection{The category $\Theta C$}

Let $C$ be a category.  We define a new category $\Theta C$ as
follows.  The objects of $\Theta C$ are tuples of the form
$([m],c_1,\dots,c_m)$, where $[m]$ is an object of $\Delta$ and
$c_1,\dots,c_m$ are objects of $C$.  It will be convenient to write
$[m](c_1,\dots,c_m)$ for this object, and to write $[0]$ for the
unique object with $m=0$.

Morphisms $[m](c_1,\dots,c_m)\ra
[n](d_1,\dots,d_m)$ are tuples $(\delta, \{f_{ij}\})$ consisting of
\begin{enumerate}
\item [(i)] a morphism $\delta \colon [m]\ra [n]$ of $\Delta$, and
\item [(ii)] for each pair $i,j$ of integers such that $1\leq i\leq
  m$, $1\leq j\leq n$, and $\delta(i-1)< j\leq \delta(i)$, a
  morphism $f_{ij}\colon c_i\ra d_j$ of $C$.
\end{enumerate}
In other words,
\[
(\Theta C)([m](c_1,\dots,c_m), [n](d_1,\dots,d_n)) \approx
\coprod_{\delta\colon [m]\ra [n]}\; \prod_{i=1}^m \;
\prod_{j=\delta(i-1)+1}^{\delta(i)} C(c_i,d_j).
\]

The composite
\[
[m](c_1,\dots,c_m)\xra{(\delta, \{f_{ij}\})} [n](d_1,\dots,d_n)
\xra{(\epsilon, \{g_{jk}\})} [p](e_1,\dots,e_p)
\]
is the pair $(\epsilon \delta, \{h_{ik}\})$, where
$h_{ik}=g_{jk}f_{ij}$ for the unique value of $j$ for which $f_{ij}$
and $g_{jk}$ are both defined.  

Pictorially, it is convenient to represent an object of $\Theta C$ as
a sequence of arrows labelled by objects of $C$.  For instance,
$[3](c_1,c_2,c_3)$ would be drawn
\[\xymatrix{
{0} \ar[r]^{c_1} & {1} \ar[r]^{c_2} & {2} \ar[r]^{c_3} & {3}
}\]
An  example of a morphism $[3](c_1,c_2,c_3)\ra [4](d_1,d_2,d_3,d_4)$
is the picture
\[\xymatrix{
& {0} \ar[r]^{c_1}="c1" \ar@{|->}[dl] & {1} \ar[r]^{c_2} \ar@{|->}[d] & {2}
\ar[r]^{c_3}="c3" \ar@{|->}[dl] & {3} \ar@{|->}[dl]
\\
{0} \ar[r]_{d_1}="d1" \ar@{~>}"c1";"d1"|{f_{11}} & {1}
\ar[r]_{d_2}="d2" \ar@{~>}"c1";"d2"|{f_{12}} & {2} \ar[r]_{d_3}="d3" \ar@{~>}"c3";"d3"|{f_{33}} & {3} 
\ar[r]_{d_4} & {4}
}\]
where the dotted arrows describe the map $\delta^{0223}\colon [3]\ra
[4]$, and the squiggly arrows represent  morphisms $f_{11}\colon
c_1\ra d_1, f_{12}\colon c_1\ra 
d_2, f_{33}\colon c_3\ra d_3$ in $C$.

Observe that (as suggested by our notation) there are functors
\[
[m]\colon C^{\times m} \ra \Theta C
\]
for $m\geq0$, defined in the evident way on objects, and which to a
morphism $(g_i\colon c_i\ra d_i)_{i=1,\dots, m}$ assign the morphism
$(\id,\{f_{ij}\})$ where $f_{ii}=g_i$.

If $C$ is a small category, then so is $\Theta C$, and it is apparent
that $\Theta$ describes a $2$-functor $\Cat\ra \Cat$.  

\subsection{A notation for morphisms in $\Theta C$}

We use the following notation for certain maps in $\Theta C$.  
Suppose $(\delta,\{f_{ij}\})\colon
[m](c_1,\dots,c_m)\ra [n](d_1,\dots,d_n)$ is a morphism in $\Theta C$
such that for each $i=1,\dots,m$, the sequence of maps
$(f_{ij}\colon c_i\ra d_j)_{j=\delta(i-1)+1,\dots,\delta(i)}$
identifies $c_i$ as the product of the $d_j$'s in $C$.  Then we simply
write $\delta$ for this morphism.  Note that even if $C$ is a category
which does not have all products, this notation is always sensible if
$\delta\in \Delta$ is injective and sequential.

\begin{rem}\label{rem:theta-morphisms-products}
If $C$ is a category with finite products, morphisms in $\Theta C$
amount to pairs $(\delta,\{f_i\}_{i=1,\dots,m})$, where 
\[
f_i\colon c_i \ra d_{\delta(i-1)+1}\times \cdots \times d_{\delta(i)}.
\]
In this case, our special notation is to write $\delta$ for
$(\delta,\{\id\}_{i=1,\dots,m})$.  

There is a variant of the $\Theta$ construction which works when $C$
is a monoidal category.  If $C$ is a monoidal category, we can define
a category $\Theta^{\mathrm{mon}}C$ with the same objects as $\Theta
C$, but with morphisms $[m](c_1,\dots,c_m)\ra [n](d_1,\dots,d_n)$
corresponding to tuples $(\delta, \{f_i\}_{i=1,\dots,m})$ where
\[
f_i\colon c_i \ra d_{\delta(i-1)+1}\otimes \cdots \otimes
d_{\delta(i)}.
\]
It seems likely that this variant notion should be useful for producing
presentations of categories enriched over general monoidal model categories.
\end{rem}

\subsection{The categories $\Theta_n$}

For $n\geq0$ we define categories $\Theta_n$ by setting $\Theta_0 =
1$ (the terminal category), and defining $\Theta_n \defeq \Theta
\Theta_{n-1}$.  One sees immediately that $\Theta_1$ is isomorphic
to $\Delta$.  

\begin{rem}
The category $\Theta_n$ can be identified as a category of finite
planar trees of level $\leq n$ \cite{joyal-theta-note}.   The opposite
category $\Theta_n^\op$ is isomorphic to the category of
``combinatorial $n$-disks'' in the sense of Joyal
\cite{joyal-theta-note}; see
\cite{cheng-lauda-illustrated-guide}*{Ch.\ 7},
  \cite{berger-iterated-wreath}. 
\end{rem}

\subsection{$\Theta$ and enriched categories}

If $V$ is a cartesian closed category and $\varnothing$ is an initial object of
$V$, it is straightforward to show that 
\begin{enumerate}
\item for every object $v\in \ob V$, the product $\varnothing\times v$ is an
  initial object of $V$, and 
\item for an object $v\in \ob V$, the set $\hom_V(v,\varnothing)$ is non-empty
  if and only if $v$ is initial.
\end{enumerate}

Suppose that $V$ is a cartesian closed category with a chosen
initial object $\varnothing$.
The \dfn{tautological
  functor} 
\[
\tau\colon \Theta V\ra \enrcat{V}
\]
is defined as follows.  For an object $[m](v_1,\dots,v_m)$, we let
$C=\tau([m](v_1,\dots,v_m))$ be the $V$-category with object set
$C_0=\{0,1,\dots,m\}$, and with morphism objects
\[
C(p,q) =
\begin{cases}
  \varnothing & \text{if $p>q$,} \\
  1 & \text{if $p=q$,} \\
  v_{p+1}\times \cdots \times v_q & \text{if $p<q$.}
\end{cases}
\]
The unique maps $1\ra C(p,p)$ define ``identity maps'', and
composition $C(p,q)\times C(q,r)\ra C(p,r)$ is defined in the evident
way.  It is clear how to define $\tau$ on morphisms.

\begin{rem}
The functor $\tau$ is not fully faithful.  For instance,
there is a $V$-functor $f\colon\tau([1](\varnothing))\ra
\tau([1](\varnothing))$ 
which on 
objects sends $0\in [1]$ to $1\in [1]$ and vice versa; this map $f$ is
not in the image of $\tau$.
\end{rem}

For a full subcategory $W$ of $V$, we will write $\tau\colon \Theta
W\ra \enrcat{V}$ for the evident composite $\Theta W\ra \Theta
V\xra{\tau} \enrcat{V}$. 
\begin{prop}[\cite{berger-iterated-wreath}*{Prop.\ 3.5}]
If $W$ is a full subcategory of $V$ which does not contain any initial
objects of $V$, then $\tau\colon \Theta W\ra \enrcat{V}$ is fully
faithful. 
\end{prop}
\begin{proof}
The fact that only an initial object can map to an initial object in
$V$ implies that for $c_i,d_j\in \ob W$, a functor
$F\colon \tau([m](c_1,\dots,c_m))\ra \tau([n](d_1,\dots,d_n))$ is necessarily
given on objects by a weakly monotone function $\delta\colon
\{0,\dots,m\} \ra \{0,\dots,n\}$.  Given $\delta$, the functor $F$
determines and is determined by morphisms $f_{ij}\colon c_i\ra d_j$
for $i=1,\dots,m$, $j=\delta(i-1)+1,\dots,\delta(i)$.  
\end{proof}

\begin{cor}\label{cor:tau-n-fully-faithful}
For each $n\geq0$, the functor $\tau_n\colon \Theta_n\ra \stcat{n}$
defined inductively as the composite
\[
\Theta_n \xra{\Theta \tau_{n-1}} \Theta (\stcat{(n\mbox{-}1)})
\xra{\tau} \stcat{n} 
\]
is fully faithful.
\end{cor}
Thus, we can identify $\Theta_n$ with a full subcategory of $\stcat{n}$.

\section{Presheaves of spaces over $\Theta C$}

In the next few sections we will be especially concerned with the
category $\sPsh(\Theta C)$ of simplicial presheaves on $\Theta C$.
In this section we describe two essential constructions.  First, we
describe an adjoint pair of functors $(T_\#,T^*)$ between simplicial
presheaves on $\Theta C$ and simplicial presheaves on $\Delta=\Theta
1$.  Next, we describe a functor $V$, called the ``intertwining
functor'', which relates $\Theta(\sPsh(C))$ and $\sPsh(\Theta C)$.  

\subsection{The functors $T^*$ and $T_\#$}
\label{subsec:the-functors-t}

Let $T\colon \Delta\ra \sPsh(\Theta C)$ be the functor defined
by
\[
(T[n])([m](c_1,\dots,c_m)) \defeq \Delta([m],[n]),
\]
Observe that if $C$ has a
terminal object $t$, then $T[n] \approx F_{\Theta C}[n](t,\dots,t)$.  

Let $T^*\colon \sPsh(\Theta C)\ra \sPsh(\Delta)$ denote the functor
defined by $(T^*X)[m] \defeq \Map_{\sPsh(\Theta C)}(T[m],X)$.  The
functor $T^*$ preserves limits, and has a left adjoint $T_\#\colon
\sPsh(\Delta)\ra 
\sPsh(\Theta C)$. 
\begin{prop}\label{prop:t-hash-formula}
On objects $X$ in $\sPsh(\Delta)$, the object $T_\#X$ is given by
\[
(T_\#X)[m](c_1,\dots,c_m) \approx  X[m].
\]
\end{prop}
\begin{proof}
A straightforward calculation.
\end{proof}
\begin{cor}\label{cor:t-hash-properties}
The functor $T_\#\colon \sPsh(\Delta)^\inj \ra \sPsh(\Theta C)^\inj$
preserves small limits, cofibrations and weak equivalences; in
particular, it is the left adjoint 
of a Quillen pair.
\end{cor}

We will regard $T^*X$ as the ``underlying simplicial space'' of the
object $X$ in $\sPsh(\Theta C)$.

\subsection{The intertwining functor $V$}
\label{subsec:intertwining}

The \dfn{intertwining functor} 
\[
V\colon \Theta(\sPsh(C)) \ra \sPsh(\Theta C)
\]
is a functor which extends the Yoneda functor $F_{\Theta C}\colon \Theta C \ra
\sPsh(\Theta C)$; it will play a crucial role in what follows.

Recall \eqref{rem:theta-morphisms-products} that since $\sPsh(C)$ has
finite 
products, a morphism $[m](A_1,\dots,A_m)\ra [n](B_1,\dots,B_n)$
in $\Theta(\sPsh(C))$ amounts to a a pair
$(\delta,\{f_j\}_{j=1,\dots,m})$ where $\delta\colon [m]\ra [n]$ in
$\Delta$, and
\[
f_j \colon A_j \ra \prod_{j=\delta(k-1)+1}^{\delta(k)} B_k \qquad
\text{in $\sPsh(C)$}.
\]

On objects $[m](A_1,\dots,A_m)$ in $\Theta(\sPsh(C))$ the functor $V$
is defined by  
\[
\bigl(V[m](A_1,\dots,A_m)\bigr)([q](c_1,\dots,c_q)) =
\coprod_{\delta\in \Delta([q],[m])\,}\,
\prod_{i=1}^q\, \prod_{j=\delta(i-1)+1}^{\delta(i)} A_j(c_i).
\]
To a morphism $(\sigma,\{f_{j}\})\colon [m](A_1,\dots,A_m)\ra
[n](B_1,\dots,B_n)$ we associate the map of presheaves defined by
\[
\coprod_{\delta\in \Delta([q],[m])} \,\prod_{i=1}^q\,
\prod_{j=\delta(i-1)+1}^{\delta(i)} A_j(c_i)
\ra \coprod_{\delta'\in \Delta([q],[n])} \,\prod_{i=1}^q\,
\prod_{k=\delta'(i-1)+1}^{\delta'(i)} B_k(c_i)
\]
which sends the summand associated to $\delta$ to the summand
associated to $\delta'=\sigma\delta$ by a map which is a product of
maps of the form $f_j(c_i)$.
 
Observe that
\begin{align*}
(V[m](Fd_1,\dots,Fd_m))([q](c_1,\dots,c_q)) &\approx 
\coprod_{\delta\colon [q]\ra [m]} \,\prod_{i=1}^q\,
\prod_{j=\delta(i-1)+1}^{\delta(i)} C(c_i,d_j) 
\\
&\approx (\Theta
C)([q](c_1,\dots,c_q),[m](d_1,\dots,d_m) )
\\
& \approx F_{\Theta C}[m](d_1,\dots,d_m)([q](c_1,\dots,c_q)).
\end{align*}
Thus we obtain a natural isomorphism $\nu\colon F_{\Theta C} \ra
V(\Theta F_C)$ of functors $\Theta C\ra \sPsh(\Theta C)$. 

In this paper, we are proposing the category $\sPsh(\Theta C)$ as a
model for $\sPsh(C)$-enriched categories.  In this light, the object
$V[m](A_1,\dots,A_m)$ of $\sPsh(\Theta C)$ may be 
thought of as a model of the $\sPsh(C)$-enriched category freely
generated by the $\sPsh(C)$-enriched graph
\[
(0)\xra{A_1} (1) \xra{A_2}\cdots \xra{A_{m-1}} (m-1) \xra{A_m} (m).
\]

The following proposition describes how the intertwining functor
interacts with colimits.  Recall that for an object $X$ of a category
$C$, $A\backslash X$ denotes the slice category of objects under $X$
in $C$.
\begin{prop}\label{prop:intertwining-functor-colimit-properties}
The intertwining functor $V\colon \Theta(\sPsh(C))\ra \sPsh(\Theta C)$
has the following properties.  Fix $m,n\geq0$ and objects
$A_1,\dots,A_m, B_1,\dots,B_n$ of $\sPsh(C)$.
\begin{enumerate}
\item
The map
$(V\delta^{0,\dots, m}, V\delta^{m+1+1,\dots,m+1+n})$ which sends
\[
V[m](A_1,\dots,A_m)\amalg V[n](B_1,\dots,B_n) \ra
V[m+1+n](A_1,\dots,A_{m},\varnothing,B_1,\dots,B_n)
\]
is an isomorphism.
\item
The functor
\begin{multline*}
X\mapsto V[m+1+n](A_1,\dots,A_m,X,B_1,\dots,B_n)\colon 
\\
\sPsh(C) \ra
V[m+1+n](A_1,\dots,A_m,\varnothing,B_1,\dots,B_n)\backslash
\sPsh(\Theta C)
\end{multline*}
is a left adjoint.
\end{enumerate}
\end{prop}
\begin{proof}
For $p=0,\dots,q+1$, let $G(p)\subseteq \Delta([q],[m+1+n])$ be
defined by
\[
G(p) =
\begin{cases}
  \set{\delta}{\delta(0)\geq m+1} & \text{if $p=0$,}
\\ 
  \set{\delta}{\delta(p-1)\leq m,\, \delta(p)\geq m+1} & \text{if $1\leq p\leq
    q$,}
\\
  \set{\delta}{\delta(q)\leq m} & \text{if $p=q+1$.}
\end{cases}
\]
Thus the $G(p)$ determine a partition of the set $\Delta([q],[m+1+n])$.  The
coproduct which defines
$V[m](A_1,\dots,A_m,X,B_1,\dots,B_n)(\theta)$ for
$\theta=[q](c_1,\dots,c_q)$ decomposes 
into factors according to this partition of $\Delta([q],[m+1+n])$.
Under this decomposition, the factor corresponding to $p=0$ is
\[
\coprod_{\delta\in G(0)}\,
\prod_{i=1}^q\,\prod_{j=\delta(i-1)+1}^{\delta(i)} B_{j-(m+1)}(c_i) \approx
V[n](B_1,\dots,B_n)(\theta),
\]
the factor corresponding to $p=q+1$ is 
\[
\coprod_{\delta\in G(q+1)}\,
\prod_{i=1}^q\, \prod_{j=\delta(i-1)+1}^{\delta(i)} A_j(c_i) \approx
  V[m](A_1,\dots, A_m)(\theta),
\]
while the factor corresponding to $p$ where $1\leq p\leq q$ is
\[
\coprod_{\delta\in G(p)}
\left(\prod_{i=1}^{p} \,\prod_{j=\delta(i-1)+1}^{\min(\delta(i),m)} A_j(c_i)
\right) \times X(c_p) \times \left(\prod_{i=p}^q \,
  \prod_{j=\max(\delta(i-1),m)+2}^{\delta(i)} B_{j-(m+1)}(c_i)\right).
\]
From this claim (1) is immediate, as is the observation that the
functor described in (2) preserves colimits, and so has a right
adjoint.  
\end{proof}

\begin{prop}
For all $m,n\geq0$ and objects $A_1,\dots,A_m$, $B_1,\dots,B_n$ in
$\sPsh(C)$, the functor
\begin{multline*}
X\mapsto V[m+1+n](A_1,\dots,A_m,X,B_1,\dots,B_n)\colon 
\\
\sPsh(C) \ra
V[m+1+n](A_1,\dots,A_m,\varnothing,B_1,\dots,B_n)\backslash
\sPsh(\Theta C)
\end{multline*}
preserves cofibrations and weak equivalences, and thus is the left
adjoint of a Quillen pair.
\end{prop}
\begin{proof}
A straightforward calculation, using the decomposition given in the
proof of \eqref{prop:intertwining-functor-colimit-properties}.
\end{proof}

\begin{rem}
It can be shown that $V$ is the left Kan extension of $F_{\Theta C}$
along $\Theta F_C$.
\end{rem}

\subsection{A product decomposition}

We will need to make use of the following description of the product
$V[1](A)\times V[1](B)$ in $\sPsh(\Theta C)$.
\begin{prop}\label{prop:product-decomposition}
The map
\[
\colim\bigl( V[2](A,B) \xla{V\delta^{02}} V[1](A\times B)
\xra{V\delta^{02}} V[2](B,A) \bigr) \ra V[1](A)\times V[1](B),
\]
induced by $(V\delta^{011},V\delta^{001})\colon V[2](A,B)\ra
V[1](A)\times V[1](B)$ and $(V\delta^{001},V\delta^{011})\colon
V[2](B,A)\ra V[1](A)\times V[1](B)$,
is an isomorphism.
\end{prop}
\begin{proof}
This is a straightforward calculation.  To grok the argument, it may
be helpful to contemplate the following diagram.
\[
\xymatrix{
{\bullet} \ar[r]^A \ar[d]_B \ar[dr]|{A\times B}
& {\bullet} \ar[d]^B
\\
{\bullet} \ar[r]_A 
& {\bullet}
}\]
\end{proof}

\subsection{Subobjects of $V[m](c_1,\dots,c_m)$}
\label{subsec:v-subobjects}

Observe that $V[m](1,\dots,1)\approx T_\#F[m]$; we write $\pi\colon
V[m](A_1,\dots,A_m)\ra T_\#F[m]$ for the map induced by projection to
the terminal object in $\sPsh(C)$.  Given a subobject $f\colon K\subseteq
F[m]$ in $\sPsh(\Delta)$, we define $V_K(A_1,\dots,A_m)$ to be the
inverse limit of the diagram
\[
V[m](A_1,\dots,A_m) \xra{\pi} T_\#F[m] \xla{T_\#f} T_\#K.
\]
Explicitly, $V_K(A_1,\dots,A_m)([q](c_1,\dots,c_q))$ is the subobject
of $V[m](A_1,\dots,A_m)([q](c_1,\dots,c_q))$ coming from summands
associated to $\delta\colon [q]\ra [m]$ such that $F\delta\colon
F[q]\ra F[m]$ factors through $K\subseteq F[m]$ in $\sPsh(\Delta)$.
Observe that $V_{F[m]}(A_1,\dots,A_m)\approx V[m](A_1,\dots,A_m)$.

These subobjects will be used in \S\ref{sec:segal-is-cartesian}.

\subsection{Mapping objects}
\label{subsec:mapping-objects}

Given an object $X$ in $\sPsh(\Theta C)$, an ordered sequence
$x_0,\dots,x_m$ of points in $X[0]$, and a sequence $c_1,\dots,c_m\in
\ob C$, we define $M_X(x_0,\dots,x_m)(c_1,\dots,c_m)$ to be the
pullback of the diagram
\[
\{(x_0,\dots,x_m)\} \ra X[0]^{\times m+1} \xla{(X\delta^0,\dots,X\delta^m)}
X[m](c_1,\dots,c_m).
\]
Allowing the objects $c_1,\dots,c_m$ to vary gives us a presheaf
$M_X(x_0,\dots,x_m)$ 
in $\sPsh(C^{\times m})$.   We will be especially interested in
$M_X(x_0,x_1)$, an object of $\sPsh(C)$, which we will refer to as a
\dfn{mapping object} for $X$.

We can use the intertwining functor $V$ to get a fancier version of the
mapping objects, as follows.  Again, given $X$ in $\sPsh(\Theta C)$
and $x_0,\dots,x_m\in X[0]$, and also given objects $A_1,\dots,A_m$ in
$\sPsh(C)$, we define $\widetilde{M}_X(x_0,\dots,x_q)(A_1,\dots,A_q)$
to be the pullback of the diagram
\[
\{(x_0,\dots,x_m)\} \ra X[0]^{\times m+1} \approx
\Map(V[m](\varnothing,\dots,\varnothing),X) \la \Map(V[m](A_1,\dots,A_m),X)
\]
where the right-hand map is induced by the maps $X\delta^i_*$.  
Allowing the objects $A_1,\dots,A_m$ to vary gives us a functor
$\widetilde{M}_X(x_0,\dots,x_m)\colon \sPsh(C)^{\times m}\ra
\sPsh(\Theta C)$.
Observe that 
\[M_X(x_0,\dots,x_m)(c_1,\dots,c_m)\approx
\widetilde{M}_X(x_0,\dots,x_m)(Fc_1,\dots,Fc_m),
\] and that
\begin{multline*}
\widetilde{M}_X(x_0,\dots,x_{m+1+n})(A_1,\dots,A_m, \varnothing,
B_1,\dots,B_n) 
\\
\approx \widetilde{M}_X(x_0,\dots,x_m)(A_1,\dots,
A_m)\times
\widetilde{M}_X(x_{m+1+1},\dots,x_{m+1+n})(B_1,\dots,B_n).
\end{multline*}

\begin{lemma}\label{lemma:mapping-space-identification}
There is a natural isomorphism
\[
\widetilde{M}_X(x_0,x_1)(A) \approx \Map_C(A, M_X(x_0,x_1)).
\]
\end{lemma}
\begin{proof}
This follows using the natural isomorphisms
\[
\widetilde{M}_X(x_0,x_1)(Fc) \approx M_X(x_0,x_1)(c) \approx
\Map_C(Fc, M_X(x_0,x_1))
\]
and the fact that $V[1]\colon \sPsh(C) \ra V[1](\varnothing)\backslash
\sPsh(\Theta C)$ preserves colimits
\eqref{prop:intertwining-functor-colimit-properties}, which implies
that $A\mapsto \tilde{M}_X(x_0,x_1)(A)$ takes colimits to limits. 
\end{proof}

\section{Segal objects}
\label{sec:segal-objects}

In this section, we examine the properties of a certain class of
objects in $\sPsh(\Theta C)$, called \dfn{Segal objects}.  In the case
that $C=1$, these are the \dfn{Segal spaces} of
\cite{rezk-ho-theory-of-ho-theory}.   We work with a fixed small
category $C$.

\subsection{Segal maps and Segal objects}
\label{subsec:the-set-se}

Let $\Se_C$ denote the set of morphisms in $\sPsh(\Theta C)$  of the form 
\[
\se^{(c_1,\dots,c_m)}\defeq (F\delta^{01},\dots, F\delta^{m-1,m})\colon
G[m](c_1,\dots,c_m)\ra 
F[m](c_1,\dots, c_m)
\]
where
\[
G[m](c_1,\dots,c_m)\defeq \colim\bigl( F[1](c_1)\xla{F\delta^1} {F[0]}
\xra{F\delta^0} \cdots 
\xla{F\delta^1} {F[0]} \xra{F\delta^0} F[1](c_m)  \bigr)
\]
for $m\geq 2$ and $c_1,\dots,c_m\in \ob C$.  It is straightforward
to check that an injective fibrant $X$ in $\sPsh(\Theta C)$ is
$\Se_C$-fibrant if and only if each of the induced maps
\[
X[m](c_1,\dots,c_m) \ra
\lim\bigl( X[1](c_1) \xra{X\delta^1} X[0]
\xla{X\delta^0} \cdots \xra{X\delta^1} X[0] \xla{X\delta^0}
X[1](c_m)\bigr)
\]
is a weak equivalence.  Equivalently, an injective fibrant $X$ is
$\Se_C$-fibrant if and only if the evident maps
\[
M_X(x_0,\dots,x_m)(c_1,\dots,c_m) \ra M_X(x_0,x_1)(c_1)\times \cdots
\times 
M_X(x_{m-1},x_m)(c_m)
\]
are weak equivalences.

A \dfn{Segal object} is a $\Se_C$-fibrant object in $\sPsh(\Theta
C)^\inj$, i.e., a fibrant object in $\sPsh(\Theta C)^\inj_{\Se_C}$. 

\subsection{More maps of Segal type}

For an object $[m](A_1,\dots,A_m)$ of $\Theta(\sPsh(C))$, we obtain a map in
$\sPsh(\Theta C)$ of the form
\[
\se^{[m](A_1,\dots,A_m)}\colon V_{G[m]}(A_1,\dots,A_m)  \ra V[m](A_1,\dots,A_m).
\]
induced by $\se^{(1,\dots,1)}\colon G[m]\ra F[m]$ in $\sPsh(\Delta)$,
where $V_{G[m]}$ is as defined in \S\ref{subsec:v-subobjects}.
Observe that 
\[
V_{G[m]}(A_1,\dots,A_m) 
\approx
\colim\bigl( V[1](A_1) \xla{\delta^1} V[0]
\xra{\delta^0} \cdots 
\xla{\delta^1} V[0] \xra{\delta^0} 
V[1](A_m)\bigr).
\]
Note also that if $\delta\colon [m]\ra [n]$ in
$\Delta$ is injective and sequential, then $F\delta\colon F[m]\ra
F[n]$ carries 
$G[m]$ into $G[n]$, and thus we obtain an induced map
\[
\delta_*\colon V_{G[m]}(A_{\delta(1)},\dots,A_{\delta(m)}) \ra
V_{G[n]}(A_1,\dots,A_m).
\]

It is straightforward to check that $V_{G[m]}\colon \sPsh(C)^{\times
  m}\ra \sPsh(\Theta C)$ satisfies formal properties
similar to $V[m]\colon \sPsh(C)^{\times m}\ra \sPsh(\Theta C)$.  Namely,
\begin{enumerate}
\item for all $A_1,\dots,A_m,B_1,\dots,B_n$ objects of $\sPsh(C)$, the
  map $(\delta^{0,\dots,m}_*, \delta^{m+1+1,\dots,m+1+n}_*)$ which
  sends
\[
V_{G[m]}(A_1,\dots,A_m)\amalg V_{G[n]}(B_1,\dots,B_n)\ra
V_{G[m+1+n]}(A_1,\dots,A_m,\varnothing, B_1,\dots,B_n)
\]
is an isomorphism, and
\item
for all $m,n\geq0$ and objects $A_1,\dots,A_m,B_1,\dots,B_n$ in
$\sPsh(C)$, the functor
\begin{multline*}
X\mapsto V_{G[m+1+n]}(A_1,\dots,A_m,X,B_1,\dots, B_n) \colon
\\
\sPsh(C)\ra V_{G[m+1+n]}(A_1,\dots,A_m,\varnothing,
B_1,\dots,B_n)\backslash \sPsh(\Theta C)
\end{multline*}
is a left Quillen functor.
\end{enumerate}

We record the following fact.
\begin{prop}\label{prop:generalized-segal-maps}
For all objects $[m](A_1,\dots, A_m)$ of $\Theta(\sPsh(C))$, we have 
\[
\se^{[m](A_1,\dots,A_m)} \in \overline{\Se_C},
\]
where $\overline{\Se_C}$ is the class of $\Se_C$-local equivalences. 
\end{prop}
\begin{proof}
We prove this by induction on $m\geq0$.
Let $\mathcal{D}$ denote the class of objects $[m](A_1,\dots,A_m)$ in
$\Theta(\sPsh(C))$ 
such that $\se^{[m](A_1,\dots,A_m)} \in \overline{\Se}_C$.  

We observe the following.
\begin{enumerate}
\item All objects of the form $[0]$ and $[1](A)$ are in $\mathcal{D}$,
  since $\se^{[0]}$ and $\se^{[1](A)}$ are isomorphisms.
\item All objects of the form  $[m](Fc_1,\dots, Fc_m)$
  for all $c_1,\dots,c_m\in   \ob C$ are in $\mathcal{D}$, since
  $\se^{[m](Fc_1,\dots,Fc_m)}=\se^{(c_1,\dots,c_m)}$.   
\end{enumerate}
For  $m\geq 1$ and $0\leq j\leq m$, let $\mathcal{E}_{m,j}$ denote the
class of 
objects of the form $[m](A_1,\dots,A_{j},Fc_{j+1},\dots, Fc_m)$ or
$[m](A_1,\dots, A_j,\varnothing, Fc_{j+2},\dots, Fc_m)$, where
$A_1,\dots, A_j$ in $\sPsh(C)$ and $c_{j+1},\dots,c_m\in \ob C$.
We need to  prove that $\mathcal{E}_{m,m}\subseteq \mathcal{D}$ for all
$m$.  Observation (1) says that this is so for $m=0$ and $m=1$, while
observation (2) says that 
$\mathcal{E}_{m,0}\subseteq \mathcal{D}$  for all $m$.  The proof will
be completed once we show that for all $m\geq2$ and $1\leq j\leq m$, 
$\mathcal{E}_{m,j-1}\subseteq \mathcal{D}$ implies
  $\mathcal{E}_{m,j}\subseteq \mathcal{D}$.

Consider the transformation $\alpha\colon G\ra H$ of functors
$\sPsh(C)\ra \sPsh(\Theta C)$ defined by the evident inclusion
\[
\alpha\colon V_{G[m]}(A_1,\dots,A_{j-1},{-},Fc_{j+1},\dots,Fc_m) \ra
V[m](A_1,\dots,A_{j-1},{-},Fc_{j+1}, \dots, Fc_m).
\]
The functors $G$ and $H$ produce left Quillen functors $\sPsh(C)^\inj\ra
G(\varnothing)\backslash \sPsh(\Theta C)^\inj$ and $\sPsh(C)^\inj \ra
H(\varnothing) \backslash\sPsh(\Theta C)^\inj$, and it is clear from
the explicit description of $V$ that $\alpha(\varnothing)$ is a
monomorphism.  Since $\mathcal{E}_{m,j-1}\subseteq \mathcal{D}$, we
have that $\alpha(\varnothing)\in \overline{\Se}_C$ and
$\alpha(Fc_j)\in \overline{\Se}_C$ for all objects $c_j$ of $C$.  Thus
\eqref{prop:left-quillen-saturations} applies to show that
$\alpha(A_j)\in \overline{\Se}_C$ for all objects $A_j$ of
$\sPsh(C)$, and thus $[m](A_1,\dots,A_j,Fc_{j+1},\cdots,Fc_m)\in
\mathcal{D}$; that is, $\mathcal{E}_{m,j}\subseteq \mathcal{D}$, as desired.
\end{proof}

\begin{cor}
Let $X$ be a $\Se_C$-fibrant object of $\sPsh(\Theta C)$, let
$x_0,\dots,x_m\in X[0]$, 
and let
$A_1,\dots, A_m$ be objects of $\sPsh(C)$.  Then the map
\[
\Map_{\Theta C}(V[m](A_1,\dots,A_m),X) \ra \Map_{\Theta
  C}(V[1](A_1),X) \times_{X[0]} \cdots \times_{X[0]} \Map_{\Theta
  C}(V[1](A_m),X)\]
induced by $V\delta^{i-1,i}$ for $1\leq i\leq m$ 
is a weak equivalence, and the map
\[
\widetilde{M}_X(x_0,\dots,x_m)(A_1,\dots,A_m) \ra
\widetilde{M}_X(x_0,x_1)(A_1)\times \cdots \times
\widetilde{M}_X(x_{m-1},x_m)(A_m)
\]
is a weak equivalence.
\end{cor}

\section{$(\Theta C, \Se_C)$ is a cartesian presentation}
\label{sec:segal-is-cartesian}

We now prove the following result.
\begin{prop}\label{prop:se-is-cartesian}
For any small category $C$, the pair $(\Theta C,\Se_C)$ is a cartesian
presentation, and thus $\sPsh(C)^\inj_{\Se_C}$ is a cartesian model
category. 
\end{prop}
Our proof is an adaptation of the proof we gave in
\cite{rezk-ho-theory-of-ho-theory}*{\S10} for the case $C=1$; it
follows after \eqref{prop:product-of-covers} below.

\subsection{Covers}

Let $[m]$ be an object of $\Delta$, and let $K\subseteq F[m]$ be a
subobject in $\sPsh(\Delta)$.  We say that $K$ is a \dfn{cover} of
$F[m]$ if 
\begin{enumerate}
\item [(i)] for all sequential $\delta\colon [1]\ra [m]$, the map
  $F\delta\colon F[1]\ra F[m]$ factors 
  through $K$, and
\item [(ii)] there exists a (necessarily unique) dotted arrow making
  the diagram commute in every diagram of the form
\[\xymatrix{
{F[1]} \ar[r] \ar[d]_{F\delta^{0n}} & {K} \ar[d]
\\
{F[n]} \ar[r] \ar@{.>}[ur] & {F[m]}
}\]
\end{enumerate}

It
is immediate that 
\begin{enumerate}
\item [(0)] the identity map $\id\colon F[m]\ra F[m]$ is a cover;
\item [(1)] the subobject $G[m]\subseteq F[m]$ generated by the images
  of the maps $F\delta^{i-1,i}\colon F[1]\ra F[m]$ is a cover (called
  the \dfn{minimal cover});
\item [(2)] if $\delta \colon [p]\ra [m]$ is sequential, and
  $K\subseteq F[m]$ is a cover, then the pullback
  $\delta^{-1}K\subseteq F[p]$ of $K$ along $F\delta$ is a cover of $F[p]$;
\item [(3)] if $\delta\colon [p]\ra [m]$ and $\delta'\colon [p]\ra
  [n]$ are sequential, and $M\subseteq F[m]$ and $N\subseteq F[n]$ are
  covers, then the pullback $(\delta,\delta')^{-1}(M\times N)$ of
  $M\times N$ along $(F\delta,F\delta')\colon F[p]\ra F[m]\times F[n]$
  is a cover of $F[p]$.
\end{enumerate}

\subsection{Covers produce $\Se_C$-equivalences}

Recall that given  a subobject $K\subseteq F[m]$ in $\sPsh(\Delta)$,
and a sequence $A_1,\dots, A_m$ of $\sPsh(C)$, we have defined (in
\S\ref{subsec:v-subobjects}) a subobject
$V_K(A_1,\dots,A_m)$ of $V[m](A_1,\dots,A_m)$ in $\sPsh(\Theta C)$.

\begin{prop}\label{prop:covers-produce-se-equivs}
If $K\subseteq F[m]$ is a cover in $\sPsh(\Delta)$, then
$V_K(A_1,\dots,A_m) \ra V[m](A_1,\dots,A_m)$ is in $\overline{\Se}_C$
for all $A_1,\dots,A_m$ objects of $\sPsh(C)$.
\end{prop}
\begin{proof}
Since $V_{G[m]}(A_1,\dots,A_m) \ra V[m](A_1,\dots,A_m)$ is in
$\overline{\Se}_C$ by \eqref{prop:generalized-segal-maps}, it suffices
to show that 
$V_{G[m]}(A_1,\dots,A_m) \ra V_K(A_1,\dots,A_m)$ is in
$\overline{\Se}_C$ for covers $K\subset F[m]$ which are proper
inclusions. We will prove this using induction on $m$. 

Given a subobject $K\subseteq F[m]$ in $\sPsh(\Delta)$, let
$\mathcal{P}_K$ denote the 
category whose objects are injective sequential maps $\delta\colon
[p]\ra [m]$ such that $F\delta$ factors through $K$, and whose
morphisms $([p]\ra [m])\ra ([p']\ra [m])$ are arrows $[p]\ra [p']$ in
$\Delta$ making the evident triangle commute.   The category
$\mathscr{P}_K$ is a poset.  For each $\delta\colon [p]\ra [m]\in
\mathcal{P}_K$ we have a natural square
\[\xymatrix{
{V_{\delta^{-1}K}(A_1,\dots,A_m)} \ar[r] \ar[d]
& {V_{G[m]}(A_1,\dots,A_m)} \ar[d]
\\
{V_{F[p]}(A_1,\dots,A_m)}  \ar[r]
& {V_K(A_1,\dots,A_m)}
}\]
Observe that since $\delta$ is a monomorphism, the map $F[p]\ra F[m]$ is
a monomorphism; we have abused notation and written $F[p]$ for this
subobject.

We have that the maps 
\[\hocolim_{\mathcal{P}_K} V_{\delta^{-1}K}(A_1,\dots,A_m)\ra
V_{G[m]}(A_1,\dots,A_m)
\]
and
\[\hocolim_{\mathcal{P}_K}
V_{F[p]}(A_1,\dots,A_m) \ra 
V_K(A_1,\dots,A_m)
\] are levelwise weak 
equivalences in $\sPsh(\Theta C)$ by \eqref{prop:hocolim-poset}, since
the corresponding maps from colimits over $\mathcal{P}_K$ are
isomorphisms. 
Since the inclusion $K\subset 
F[m]$ is proper, $p<m$ for all objects of $\mathcal{P}_K$, and so each
$V_{\delta^{-1}K}(A_1,\dots,A_m)\ra V_{F[p]}(A_1,\dots,A_m) \in
\overline{\Se}_C$ by the induction 
hypothesis, the result follows using \eqref{prop:saturations-hocolims}.
\end{proof}

\subsection{Proof that $(\Theta C,\Se_C)$ is cartesian}

\begin{prop}\label{prop:product-of-covers}
If $M\subseteq F[m]$ and $N\subseteq F[n]$ are covers in
$\sPsh(\Delta)$, then 
\[
V_M(A_1,\dots,A_m)\times V_N(B_1,\dots,B_n)\ra
V[m](A_1,\dots,A_m)\times V[n](B_1,\dots,B_m)
\]
is an $\Se_C$-equivalence.
\end{prop}
The proof of \eqref{prop:product-of-covers} is deferred to the end of
this section \S\ref{subsec:proof-of-product-of-covers}.

Now we can give the proof of the proposition stated at the beginning
of the section. 

\begin{proof}[Proof of \eqref{prop:se-is-cartesian}]
To prove that $(\Theta C,\Se_C)$ is cartesian, it is enough to show
that $\se^{(c_1,\dots,c_m)}\times F\theta\colon
G[m](c_1,\dots,c_m)\times F\theta \ra F[m](c_1,\dots,c_m)\times
F\theta$ is in $\overline{\Se}_C$ for all $m\geq2$, $c_1,\dots,c_m\in
\ob C$, and $\theta\in \ob \Theta C$.  This is a special case of
\eqref{prop:product-of-covers}. 
\end{proof}

\subsection{Presentations of the form $(\Theta C, \Se_C\cup \mathscr{U})$}

Let $\mathscr{U}$ be a set of morphisms in $\sPsh(\Theta C)$. 

\begin{prop}\label{prop:presentations-segal-plus}
The presentation $(\Theta C,\Se_C \cup \mathscr{U})$ is cartesian if
and only if for all $X$ in $\sPsh(\Theta C)$ which are $(\Se_C\cup
\mathscr{U})$-fibrant, and for all $c\in \ob C$, the function object
$X^{F[1](c)}$ is $\mathscr{U}$-local.
\end{prop}
\begin{proof}
By \eqref{prop:cart-pres-equiv-charac} and
\eqref{prop:se-is-cartesian}, it is enough to show that if 
$X$ is $(\Se_C\cup \mathscr{U})$-fibrant, then 
$X^{F\theta}$ is $\mathscr{U}$-local for all
$\theta \in \ob \Theta C$.  Since $X$ is $\Se_C$-local, every
$X^{F\theta}$ is weakly equivalent to a homotopy fiber product of the
form $X^{F[1](c_1)}\times_{X}\cdots\times_{X} X^{F[1](c_m)}$, and thus
the result follows.
\end{proof}

\subsection{Proof of \eqref{prop:product-of-covers}}
\label{subsec:proof-of-product-of-covers}

The proof of \eqref{prop:product-of-covers} given in the published
version of this paper was incorrect, as was pointed out to me by Bill
Dwyer, who also suggested the idea which forms the basis of the proof
given here.  I thank Bill Dwyer for his help.

Let $\pQ_{m,n}$ denote the
category whose objects are pairs of maps $\delta=(\delta_1\colon [p]\ra
[m],\delta_2\colon [p]\ra [n])$ in $\Delta$ such that $\delta_1$ and
$\delta_2$ are surjective, and $(F\delta_1,F\delta_2)\colon F[p]\ra
F[m]\times F[n]$ is a monomorphism in $\sPsh(\Delta)$. 

Functorially associated to an 
object $\delta$ of $\pQ_{m,n}$ is the following diagram.
\begin{equation}\label{eq:main-square}
\vcenter{\xymatrix{
{V_{(\delta_1,\delta_2)^{-1}M\times N}(D_1,\dots,D_p)} \ar[r]^-{g}
\ar[d]_f 
& {V_M(A_1,\dots,A_m)\times V_N(B_1,\dots,B_n)} \ar[d]^{f'}
\\
{V[p](D_1,\dots,D_p)} \ar[r]_-{g'}
& {V[m](A_1,\dots,A_m)\times V[n](B_1,\dots,B_n)} 
}}
\end{equation}
The objects in the middle row are defined using the ``intertwining
functor'' of \S\ref{subsec:intertwining}, and of the top row
are defined using the construction of \S\ref{subsec:v-subobjects}.
The objects $C_i$ are described by
\[
D_i =
\begin{cases}
  A_{\delta_1(i)} & \text{if $\delta_1(i)>\delta_1(i-1)$ and
    $\delta_2(i)=\delta_2(i-1)$,}
\\
  B_{\delta_2(i)} & \text{if $\delta_1(i)=\delta_1(i-1)$ and
    $\delta_2(i)>\delta_2(i-1)$,} 
\\ 
  A_{\delta_1(i)}\times B_{\delta_2(i)} & \text{if
    $\delta_1(i)>\delta_1(i-1)$ and $\delta_2(i)>\delta_2(i-1)$.}
\end{cases}
\]
The maps marked $f$ and $f'$ are induced by the inclusions $M\times
N\ra F[m]\times F[n]$ and $(\delta_1,\delta_2)^{-1}M\times N\ra F[p]$. 
The maps marked $g$ and $g'$ are induced by the map
$(F\delta_1,F\delta_2)$. 

By \eqref{prop:covers-produce-se-equivs}, the map $f$ is an
$\Se_C$-equivalence, and so \eqref{prop:product-of-covers} follows
from the following 
proposition \eqref{prop:hocolim-wes}, in which homotopy colimits are
computed with respect to the 
level-wise weak equivalences in $\sPsh(\Theta C)$.
\begin{prop}\label{prop:hocolim-wes}
The induced maps $\hocolim_{\pQ_{m,n}} V_{(\delta,\delta')^{-1}M\times
  N} \ra V_M(A_1,\dots,A_m)\times V_N(B_1,\dots,B_n)$ and
$\hocolim_{\pQ_{m,n}} V[p](D_1,\dots,D_p) \ra
V[m](A_1,\dots,A_m)\times V[n](B_1,\dots,B_n)$ are levelwise weak
equivalences. 
\end{prop}

Fix an object $\theta=[q](c_1,\dots,c_q)$ of $\Theta C$.  Given a pair
of maps $\alpha=(\alpha_1\colon [q]\ra [m],\alpha_2\colon [q]\ra
[n])$, let 
\[
D(\alpha)= \prod_{i=1}^q\left( \prod_{j=\alpha_1(i-1)+1}^{\alpha_1(i)}
  A_j \times \prod_{k=\alpha_2(i-1)+1}^{\alpha_2(i)} B_k\right),
\]
an object of $\sPsh(\Theta C)$.  
For $\delta$ in $\pQ_{m,n}$ define $G_\delta(\alpha)=
\set{\gamma\colon
  [q]\ra[p]}{\delta_1\gamma=\alpha_1,\delta_2\gamma=\alpha_2}$.  Since
$(F\delta_1,F\delta_2)$ is a monomorphism, the set $G_\delta(\alpha)$
is either empty or is a singleton.  By examinining the definitions, we
see that the map $g'$ of
\eqref{eq:main-square}, evaluated at the object $\theta$, is
is the map
\[
g'(\theta)\colon \coprod_\alpha D(\alpha)\times G_\delta(\alpha) \ra
\coprod_\alpha 
D(\alpha)
\]
induced by the projections $G_\delta(\alpha)\ra 1$.  Both coproducts
are 
taken over all pairs $\alpha=(\alpha_1,\alpha_2)$.  Similarly, the map
$g$ of \eqref{eq:main-square}, evaluated at the object $\theta$ is
either (i) isomorphic to $g'(\theta)$, if
$(F\delta_1,F\delta_2)\colon F[p]\ra F[m]\times F[n]$ factors through
the subobject $M\times N$, or (ii) is an isomorphism between empty
spaces, otherwise. 

Thus we have reduced the proof of \eqref{prop:hocolim-wes} to the
following.
\begin{lemma}\label{lemma:hocolim-contractible}
For each pair $\alpha=(\alpha_1\colon[q]\ra [m],\alpha_2\colon [q]\ra
[n])$, the simplicial set $\hocolim_{\pQ_{m,n}} G_\delta(\alpha)$ is
weakly 
contractible.  
\end{lemma}

In what follows, it will be useful to view elements of the poset
$\pQ_{m,n}$ as paths $(x_0,y_0),\dots,(x_r,y_r)$ in the set $\Z\times
\Z$, which (i) start at $(x_0,y_0)=(0,0)$, (ii) end at 
$(x_r,y_r)=(m,n)$, and are such that (iii) each step in the path moves
one unit left, one unit up, or diagonally up and left by one unit
(i.e., $(x_k-x_{k-1},y_k-y_{k-1})$ is one of $(1,0)$, $(0,1)$, or
$(1,1)$).  The order relation is given by inclusion of paths.

\begin{proof}[Proof of \eqref{lemma:hocolim-contractible}]
Observe that if $\beta\colon [q']\ra [q]$ is a surjection in $\Delta$,
then $G_\delta(\alpha)\approx G_\delta(\alpha\beta)$.  Thus, 
we may assume without loss of
generality that $(F\alpha_1,F\alpha_2)\colon F[q]\ra F[m]\times F[n]$
is injective.

Let $\pQ_{m,n,\alpha}$ denote the subposet of $\pQ_{m,n}$ consisting
of $\delta$ such that $G_\delta(\alpha)$ is non-empty.  It is
immediate that $\hocolim_{\pQ_{m,n}}G_\delta(\alpha)$ is isomorphic to
the nerve of $\pQ_{m,n,\alpha}$.  

Next, observe that the subposet $\pQ_{m,n,\alpha}$ may be identified
as the subposet of paths in $\pQ_{m,n}$ which pass through the points
$(\alpha_1(j),\alpha_2(j))$ for $j=0,\dots, q$.   Therefore, there is
an isomorhpism of posets $\pQ_{m,n,\alpha} \approx \prod_{j=0}^{q+1}
\pQ_{m_j,n_j}$, where
$(m_j,n_j)=(\alpha_1(j)-\alpha_1(j-1),\alpha_2(j)-\alpha_2(j-1))$ if
$j=1,\dots,q$, and $(m_0,n_0)=(\alpha_1(0),\alpha_2(0))$ and
$(m_{q+1},n_{q+1})=(m-\alpha_1(q),n-\alpha_q(q))$.  Thus, since nerve
commutes with products, the result follows by \eqref{lemma:nerve-contr}.
\end{proof}

\begin{lemma}\label{lemma:nerve-contr}
The simplicial nerve of $\pQ_{m,n}$ is weakly contractible.
\end{lemma}
\begin{proof}
When $n=0$, $\pQ_{m,0}$ is the terminal category, and thus clearly has
contractible nerve.  Thus, it will suffice to produce, for each
$n\geq1$, a chain of inclusions
\[
\pQ_{m,n-1}\approx X_m\ra Y_{m-1}\ra X_{m-1}\ra Y_{m-2}\ra \cdots \ra Y_0\ra
X_0=\pQ_{m,n},
\]
such that each inclusion induces a simplicial homotopy equivalence on
nerves. 

For $0\leq k\leq m$ we let $X_k\subset \pQ_{m,n}$ be the subposet of
$\pQ_{m,n}$ consisting of paths that 
do not contain the points $(i,n)$ for $i<k$.  We let $Y_k\subset X_k$
be the subposet consisting of those paths such that \emph{if}
$(k,n)$ is in the path, \emph{then} $(k,n-1)$ is also in the path;
that is, the diagonal move 
from $(k-1,n-1)$ to $(k,n)$ does not occur in the path.  Thus
$X_0=\pQ_{m,n}$, and $X_m$ is isomorphic to $\pQ_{m,n-1}$.

For each $k=0,\dots,m-1$, there is a retraction $r\colon X_k\ra Y_k$,
which sends a path to a modified version of that path, in which the
diagonal move $(k-1,n-1)$ to 
$(k,n)$ (if it is present in the path) is replaced by two moves
$(k-1,n-1)$ to $(k,n-1)$ to $(k,n)$.  Note that for any element $p\in
X_k$, we have $r(p)\leq p$, and for any $p\in Y_k$, $r(p)=p$.  This
defines a simplicial homotopy retraction for the inclusion $NY_k\ra
NX_k$. 

For each $k=0,\dots,m-1$ there is a retraction $r'\colon Y_k\ra
X_{k+1}$, which sends a path to a modified version of that path, in
which the two moves $(k,n-1)$ to $(k,n)$ to $(k+1,n)$ (if present in
the path), are replaced by the diagonal move $(k,n-1)$ to $(k+1,n)$.
Note that for any element $p\in Y_k$, we have $r'(p)\geq p$, and for
any $p\in X_k$, $r'(p)=p$.  Thus $r'$ defines a simplicial homotopy
retraction for the inclusion $NX_{k+1}\ra NY_k$.
\end{proof}

\section{Complete Segal objects}

In this section, we examine the properties of a certain class of Segal
objects in $\sPsh(\Theta C)$, called \dfn{complete Segal objects}.  In
the case that $C=1$, these are precisely the \dfn{complete Segal
  spaces} of \cite{rezk-ho-theory-of-ho-theory}.  We show below that
complete Segal objects are the fibrant objects of a cartesian
model category, generalizing a result of
\cite{rezk-ho-theory-of-ho-theory}*{\S12}.

\subsection{The underlying Segal space of a Segal object}

Recall the Quillen pair $T_\#\colon \sPsh(\Delta)\rightleftarrows
\sPsh(\Theta C)\noloc T^*$ of \S\ref{subsec:the-functors-t}.  Given an
object $X$ of $\sPsh(\Theta C)$, we will call $T^*X$ in
$\sPsh(\Delta)$ its \dfn{underlying simplicial space}; according to
the following proposition, it is reasonable to call $T^*X$ the
\dfn{underlying Segal space} of $X$ if $X$ is itself a Segal object.

\begin{prop}
If $X$ is an $\Se_C$-fibrant object in $\sPsh(\Theta C)$, then $T^*X$
is an $\Se_1$-fibrant object in $\sPsh(\Theta C)=\sPsh(\Delta)$.  That
is, $T^*X$ is a Segal space in the sense of
\cite{rezk-ho-theory-of-ho-theory}. 
\end{prop}
\begin{proof}
The map 
\[
\se^{[m](1,\dots,1)} \colon \colim\bigl( V[1](1)\la V[0]\ra \cdots \la
V[0] \ra V[1](1)\bigr) \ra V[m](1,\dots,1)
\]
is isomorphic to $T_\#\se^{(1,\dots,1)}\colon T_\#G[m]\ra T_\#F[m]$.
\end{proof}

\subsection{The homotopy category of a
  Segal object}
\label{subsec:homotopy-category-segal-object}

Recall that if $X$ in $\sPsh(\Delta)$ is a Segal space, then we define
its homotopy category $hX$ as follows.  The objects of $hX$ are points
of $X[0]$, and morphisms are given by
\[
hX(x_0,x_1) \defeq \pi_0 M_X(x_0,x_1).
\]
It is shown in \cite{rezk-ho-theory-of-ho-theory} that $hX$ is indeed
a category; composition is defined and its properties verified using
the isomorphisms $\pi_0 M_X(x_0,\dots,x_m)\approx hX(x_0,x_1)\times
\dots \times hX(x_{m-1},x_m)$ which hold for a Segal space.

For an $\Se_C$-fibrant object $X$ in $\sPsh(\Theta C)$, we define its
\dfn{homotopy category} $hX$ to be the homotopy category of $T^*X$.
Explicitly, objects of $hX$ are points in $X[0]$, and morphisms are
\[
hX(x_0,x_1) \defeq \pi_0 \Gamma_C M_X(x_0,x_1).
\]

\subsection{The enriched homotopy category of a Segal object}
\label{subsec:enriched-homotopy-cat-segal-obj}

The homotopy category $hX$ described above can be refined to a
homotopy category enriched over the homotopy category $h\sPsh(C)$ of
presheaves of spaces on $C$.  This \dfn{$h\sPsh(C)$-enriched homotopy
  category} is denoted $\underline{hX}$, and is defined as follows.

We take $\ob \underline{hX}=\ob hX$.  Given objects $x_0,x_1$ of $hX$,
recall that the \dfn{function object} of maps 
$x_0\ra x_1$ is the object $M_X(x_0,x_1)$ of $\sPsh(C)$.  
For objects $x_0,x_1$
in $hX$, we set $\underline{hX}(x_0,x_1) \defeq M_X(x_0,x_1)$ viewed
as an object in the homotopy category $h\sPsh(C)$.  To make this a category, let
$\Delta_m\colon C\ra C^{\times m}$ denote the ``diagonal'' functor, and
let $\Delta^*_m\colon \sPsh(C^{\times m}) \ra \sPsh(C)$ denote the
functor which sends $F\mapsto F\Delta_m$.   The functor $\Delta^*_m$
preserves weak equivalences and products.  Observe that since $X$ is a
Segal object, there are evident weak equivalences
\[
\Delta_m^*M_X(x_0,\dots,x_m) \ra M_X(x_0,x_1)\times \cdots \times
M_X(x_{m-1},x_m)
\]
in $\sPsh(C)$.  Thus we obtain ``identity'' and ``composition'' maps
\[
1\approx \Delta_0^*M_X(x_0)\ra M_X(x_0,x_0), 
\]
\[
M_X(x_0,x_1)\times M_X(x_1,x_2)\xla{\sim} \Delta^*_2 M_X(x_0,x_1,x_2)
\ra M_X(x_0,x_2)
\]
in $h\sPsh(C)$, and it is straightforward to check that these make
$\underline{hX}$ into 
an $h\sPsh(C)$-enriched category.  Furthermore, we see that
$h\sPsh(C)(1, \underline{hX}(x_0,x_1))\approx hX(x_0,x_1)$.

\subsection{Equivalences in a Segal object}

Recall that if $X$ in $\sPsh(\Delta)$ is a Segal space, then we say
that a point in $X[1]$ is an \dfn{equivalence} if it projects to an
isomorphism in the homotopy category $hX$.  We write
$M_X^\hequiv(x_0,x_1)$ for the subspace of $M_X(x_0,x_1)$ consisting
of path components which project to isomorphisms in $hX$, and we let
$X^\hequiv$ denote the subspace of $X[1]$ consisting of path
components which contain points from $M_X^\hequiv(x_0,x_1)$ for some
$x_0,x_1\in \ob hX$.  Thus $M_X^\hequiv(x_0,x_1)$ is the fiber of
$(X\delta^0,X\delta^1)\colon X^\hequiv\ra X[0]\times X[0]$ over
$(x_0,x_1)$; observe that the map $X\delta^{00}\colon X[0]\ra X[1]$
factors through $X^\hequiv\subseteq X[1]$.

These definitions transfer to Segal objects.  Thus, if $X$ is a Segal
object in $\sPsh(\Theta C)$, we say that a point in $\Gamma_C M_X(x_1,x_2)$
is an \dfn{equivalence} if 
it projects to an isomorphism in the homotopy category $hX$.  We
define $M^\hequiv_X(x_0,x_1)$ to be the subspace of
$\Gamma_CM_X(x_0,x_1)$ consisting of path components which project to
isomorphisms in $hX$.
The
space of equivalences $X^\hequiv$ is defined to be the subspace of
$\Gamma_C X[1]$ 
consisting of
path components which contain points from $M_X^\hequiv(x_0,x_1)$ for
some $x_0,x_1\in X[0]$; thus $M^\hequiv_X(x_0,x_1)$ is the fiber of
$X^\hequiv\ra X[0]\times X[0]$ over $(x_0,x_1)$.  Observe that the map
$X\delta^{00}\colon X[0]\ra \Gamma_C X[1]$ 
factors through $X^\hequiv\subseteq \Gamma_C X[1]$.

\subsection{The set $\Cpt_C$}
\label{subsec:the-set-cpt}

Let $E$ be the object in $\sPsh(\Delta)$ which is the ``discrete nerve
of the free-standing isomorphism'', as in
\cite{rezk-ho-theory-of-ho-theory}*{\S6}.  Let $p\colon E\ra F[0]$ be
the 
evident projection map, and let $i \colon F[1]\ra E$ be the inclusion
of one of the non-identity arrows.  We recall the following result.
\begin{prop}[\cite{rezk-ho-theory-of-ho-theory}*{Thm.\ 6.2}]
\label{prop:e-represents-equiv-segal-space}
If $X$ in $\sPsh(\Delta)$ is a Segal space, then the map
$\Map(i,X)\colon \Map(E,X)\ra
\Map(F[1],X)\approx X[1]$ factors through $X^\hequiv\subseteq X[1]$
and induces a weak equivalence $\Map(E,X)\ra X^\hequiv$.
\end{prop}

We define $\Cpt_C$ to be the set consisting of the single map 
\[
T_\#p\colon T_\#E\ra T_\# F[0].
\]
We say that $X$ in $\sPsh(C)$ is a \dfn{complete Segal object} if it
is $(\Se_C\cup \Cpt_C)$-fibrant.  As a consequence of
\eqref{prop:e-represents-equiv-segal-space}, we have the following.

\begin{prop}\label{prop:e-represents-equiv-segal-object}
Let $X$ be a Segal object of $\sPsh(\Theta C)$.
The map $\Map(T_\#E,X) \ra \Map(T_\#F[1],X)\approx T^*X[1]$ factors
through $X^\hequiv\subseteq T^*X[1]$, and induces a weak equivalence
$\Map(T_\#E,X)\ra X^\hequiv$ of spaces.  Thus, a Segal object $X$ is a
complete Segal object if and only if $X[0]\ra X^\hequiv$ is a weak
equivalence of spaces.
\end{prop}

\begin{rem}
In \S\ref{sec:alternate-cs} we give another formulation of the
completeness condition, in which 
the simplicial space $E$ is replaced by a smaller one $Z$, so that
variants of \eqref{prop:e-represents-equiv-segal-space} and
\eqref{prop:e-represents-equiv-segal-object} hold with $E$ replaced by
$Z$.  Either formulation works just as well for our purposes.
\end{rem}

\subsection{Fully faithful maps}

If $X$ and $Y$ are Segal objects in $\sPsh(\Theta C)$, we
say that a map $f\colon X\ra Y$ is \dfn{fully faithful} if for all
$c\in \ob C$ the square
\[\xymatrix{
{X[1](c)} \ar[r] \ar[d] & {Y[1](c)} \ar[d]
\\
{X[0]\times X[0]} \ar[r] & {Y[0]\times Y[0]}
}\]
is a homotopy pullback square.
\begin{prop}
Let $f\colon X\ra Y$ be a map between Segal objects in $\sPsh(\Theta
C)$.  The following are equivalent.
\begin{enumerate}
\item $f$ is fully faithful.
\item For all $c\in \ob C$ and all $x_0,x_1$ points of $X[0]$, the map
  $M_X(x_0,x_1)(c)\ra   M_X(fx_0,fx_1)(c)$ induced by $f$ is a weak
  equivalence of spaces
\item The induced map $\underline{hX}\ra \underline{hY}$ of enriched
  homotopy categories is fully faithful, i.e.,
  $\underline{hX}(x_0,x_1)\ra \underline{hY}(x_0,x_1)$ is an
  isomorphism in $h\sPsh(C)$ for all points $x_0,x_1$ of $X[0]$.
\end{enumerate}
\end{prop}

\begin{prop}\label{prop:obj-to-mor-fully-faithful}
Suppose $X$ is a Segal object in $\sPsh(C)$.  Then the map $X\approx
X^{T_\#F[0]}\ra 
X^{T_\#F[1]}$ induced by $T_\#F\delta^{00}$ is fully faithful.
\end{prop}
\begin{proof}
  Observe that $T_\#F[1]\approx V[1](1)$, and that the statement will
  be proved if we can show that for all $c\in \ob C$, the square
  obtained by applying $\Map_{\Theta C}(\mbox{-},X)$ to the square
\[\xymatrix{
{V[1](Fc)} 
& {V[1](Fc)\times V[1](1)} \ar[l]_-{\text{proj}}
\\
{V[1](\varnothing)} \ar[u]^{V[1](\text{incl})}
& {V[1](\varnothing)\times V[1](1)} \ar[l]^-{\text{proj}}
\ar[u]_{V[1](\text{incl})\times \id}
}\]
is a homotopy pullback of spaces.  Using the product decomposition
\eqref{prop:product-decomposition}, we 
obtain a diagram
\begin{equation}
\label{eq:prod-decomp-for-fully-faithful-lemma}
\vcenter{
\xymatrix{
{V[2](Fc, 1)} 
& {V[1](Fc\times 1)} \ar[l]_{V\delta^{02}} \ar[r]^{V\delta^{02}}
& {V[2](1,Fc)}
\\
{V[2](\varnothing, 1)} \ar[u] \ar[d]_{V\delta^{011}}
& {V[1](\varnothing\times 1)} \ar[l]_-{V\delta^{02}}
\ar[r]^-{V\delta^{02}} \ar[u] \ar[d] 
& {V[2](1,\varnothing)} \ar[u] \ar[d]^{V\delta^{112}}
\\
{V[1](\varnothing)}
& {V[1](\varnothing)} \ar[l] \ar[r]
& {V[1](\varnothing)}
}}
\end{equation}
in which taking colimits of rows provides the diagram
\[
V[1](Fc)\times V[1](1) \xla{V[1](\text{incl})\times \id}
V[1](\varnothing)\times V[1](1) \xra{\text{proj}} V[1](\varnothing).
\]
The horizontal morphisms of
\eqref{eq:prod-decomp-for-fully-faithful-lemma} are monomorphisms, so
the colimits of the 
rows are in fact homotopy colimits in $\sPsh(\Theta C)^\inj$.  Thus,
it suffices to show that $\Map_{\Theta C}(V[1](Fc), X)$ maps by a
weak equivalence to the homotopy limit of  $\Map_{\Theta
  C}(\mbox{-},X)$ applied to the above diagram.  We claim that in fact
that the evident projection maps induce weak equivalences from
$\Map_{\Theta C}(V[1](Fc),X)$ to the inverse limits of each of the
columns of $\Map_{\Theta C}(\mbox{-},X)$ applied to the diagram. 

This is clear for the middle column: the map $V[1](\varnothing\times
1)\ra V[1](\varnothing)$ is an isomorphism, so the colimit of the
middle column 
is isomorphic to $V[1](Fc)$.  We will show the proof for the
left-hand column, leaving the right-hand column for the reader.
Consider the diagram
\[\xymatrix{
{V[1](Fc)} \ar[r]^-{V\delta^{01}}
& {V[2](Fc,1)} \ar[r]^{V\delta^{011}}
& {V[1](Fc)} 
\\
& {V[2](\varnothing, 1)}  \ar[r] \ar[u]
& {V[1](\varnothing)}\ar[u]
\\
{V[0]} \ar[r]_{V\delta^1} \ar[uu]^{V\delta^1}
& {V[1](1)} \ar[r] \ar[u]^{V\delta^{12}}
& {V[0]}\ar[u]_{V\delta^1}
}\]
We want to show that $\Map_{\Theta C}({-},X)$ carries the upper-right
square to a homotopy pullback.
The lower-right square is a pushout square (use the isomorphism
$V[2](\varnothing, 1)\approx V[0]\amalg V[1](1)$), as is the outer square;
thus they are homotopy pushouts
(of spaces)
since the vertical maps are monomorphisms.  Applying $\Map_{\Theta
  C}({-},X)$ to the diagram takes these two squares to homotopy
pullbacks of spaces; this
operation also takes the left-hand rectangle to a homotopy pullback
of spaces, since $X$ is a Segal object.  Thus we can conclude that
$\Map_{\Theta C}({-},X)$ carries the
upper-right square to a homotopy pullback of spaces, as desired.
\end{proof}

Say that a map $f\colon X\ra Y$ of spaces is a \dfn{homotopy
  monomorphism} if it is injective on $\pi_0$, and induces a weak
equivalence between each path component of $X$ and the corresponding
path component of $Y$.  Say a map $f\colon X\ra Y$ of objects of
$\sPsh(C)$ is a homotopy monomorphism if each $f(c)\colon X(c)\ra
Y(c)$ is a homotopy monomorphism of spaces.

\begin{lemma}\label{lemma:hequiv-to-mor-homono}
If $X$ is a Segal object in $\sPsh(\Theta C)$, then
$X^{T_\#i}\colon X^{T_\#E}\ra X^{T_\#F[1]}$ is a 
homotopy monomorphism in $\sPsh(\Theta C)$. 
\end{lemma}
\begin{proof}
For $\theta \in \ob \Theta C$, the map $X^{T_\#E}(\theta)\ra
X^{T_\#F[1]}(\theta)$ is isomorphic to 
\[
\Map(T_\#E,X^{F\theta}) \ra
\Map(T_\#F[1],X^{F\theta}),
\]
which since $X^{F\theta}$ is a Segal object, is weakly equivalent to
the map $(X^{F\theta})^\hequiv \ra T^*(X^{F\theta})[1]$, which is a
homotopy monomorphism.
\end{proof}

\begin{prop}\label{prop:obj-to-equiv-fully-faithful}
Suppose $X$ is a Segal object in $\sPsh(\Theta C)$.  Then the map
$X^{T_\#q}\colon X\approx X^{T_\#F[0]}\ra 
X^{T_\#E}$ is fully faithful.
\end{prop}
\begin{proof}
It is straightforward to check that if $X\xra{f} Y\xra{g} Z$ are maps
of Segal objects in $\sPsh(C)$ such that $gf$ is fully faithful and
$g$ is a homotopy monomorphism, then $f$ is fully faithful.  Apply
this observation to $X\ra X^{T_\#E}\ra X^{T_\#F[1]}$, using
\eqref{prop:obj-to-mor-fully-faithful} and \eqref{lemma:hequiv-to-mor-homono}.
\end{proof}

\subsection{Essentially surjective maps}

If $X$ and $Y$ are Segal objects in $\sPsh(\Theta C)$, we say that a
map $f\colon X\ra Y$ is \dfn{essentially surjective} if the induced
functor $hf\colon hX\ra hY$ on homotopy categories
(\S\ref{subsec:homotopy-category-segal-object}) is essentially
surjective, i.e., if every object of $hY$ is isomorphic to an object
in the image of $hf$.

\begin{prop}\label{prop:obj-to-equiv-ess-surj}
Suppose $X$ is a Segal object in $\sPsh(\Theta C)$.  Then the map
$X^{T_\#q}\colon X\approx X^{T_\#F[0]} \ra X^{T_\#E}$ is essentially
surjective. 
\end{prop}
\begin{proof}
Observe that since $T_\#$ preserves products
\eqref{cor:t-hash-properties}, the map $T^*(X^{T_\#q})\colon
T^*(X^{T_\#F[0]})\ra T^*(X^{T_\#F[1]})$ is isomorphic to the map
$(T^*X)^q\colon T^*X\approx (T^*X)^{F[0]} \ra (T^*X)^{E}$.  Thus we
are reduced to the case when $C=1$, and $X$ is a Segal space, in which
case the result follows from \cite{rezk-ho-theory-of-ho-theory}*{Lemma
  13.9}. 
\end{proof}

\begin{lemma}\label{lemma:ff-if-factor-thru-ffes}
If $X\xra{f} Y\xra{g} Z$ are maps of Segal objects in $\sPsh(\Theta
C)$ such that (i) $gf$ is fully faithful and (ii) $f$ is fully
faithful and essentially
surjective, then $g$ is fully faithful.
\end{lemma}
\begin{proof}
We need to show for all $y_0,y_1$ points of $Y[0]$ that
$\underline{hY}(y_0,y_1) \ra \underline{hZ}(gy_0,gy_1)$ is an
isomorphism in $h\sPsh(C)$.  Since $f$ is essentially surjective, we
may choose points $x_0,x_1$ in $X[0]$ so that $fx_i\approx y_i$,
$i=0,1$, as objects of $hY$.  
\end{proof}

\begin{prop}\label{prop:equiv-to-mor-fully-faithful}
If $X$ is a Segal object in $\sPsh(C)$, then the map $X^{T_\#i}\colon
X^{T_\#E}\ra X^{T_\#F[1]}$ is fully faithful.
\end{prop}
\begin{proof}
Apply \eqref{lemma:ff-if-factor-thru-ffes} to $X\ra X^{T_\#E}\ra
X^{T_\#F[1]}$, using \eqref{prop:obj-to-mor-fully-faithful},
\eqref{prop:obj-to-equiv-fully-faithful}, and
\eqref{prop:obj-to-equiv-ess-surj}. 
\end{proof}

\subsection{$(\Theta C,\Se_C\cup \Cpt_C)$ is a cartesian presentation}

\begin{prop}\label{prop:se-plus-cpt-is-cartesian}
$(\Theta C, \Se_C\cup \Cpt_C)$ is a cartesian presentation.  
\end{prop}
\begin{proof}
By \eqref{prop:presentations-segal-plus}, it suffices to show that if
$X$ is a complete Segal object, then $X^{F[1](c)}$ is $\Cpt_C$-local
for all $c\in \ob C$.  That is, we must show that $X^{F[1](c)}\approx
\Map_{\Theta 
  C}(T_\#F[0],X^{F[1](c)}) \ra \Map_{\Theta C}(T_\#E,X^{F[1](c)})$ is
a weak equivalence of spaces, or equivalently that 
\[
(X^{T_\#F[0]})[1](c)\ra (X^{T_\#E})[1](c)
\]
is a weak equivalence of spaces.  This is immediate from the fact that
$X^{T_\# F[0]}\ra X^{T_\# E}$ is fully faithful
\eqref{prop:obj-to-equiv-fully-faithful} and the 
fact that $X[0]\approx X^{T_\#F[0]}[0]\ra X^{T_\#E}[0]\approx
X^\hequiv$ is a weak equivalence, since $X$ is a complete Segal object.
\end{proof}

\section{The presentation $(\Theta C, \mathscr{S}_\Theta)$}
\label{sec:presentation-theta-css}

In this section, we consider what happens when we start with a
presentation $(C,\mathscr{S})$.  In this case, we define a new
presentation $(\Theta C,\mathscr{S}_\Theta)$ which depends on
$(C,\mathscr{S})$, by 
\[
\mathscr{S}_{\Theta} \defeq \Se_C \cup \Cpt_C \cup V[1](\mathscr{S}),
\]
where $V[1](\mathscr{S}) \defeq \set{V[1](f)}{f\in \mathscr{S}}$, and 
where $\Se_C$, and $\Cpt_C$ are as defined in
\S\ref{subsec:the-set-se} and \S\ref{subsec:the-set-cpt}.

Say that two model categories $M_1$ and $M_2$ are \dfn{equivalent} if
there is an equivalence $E\colon M_1\ra M_2$ of categories which
preserves and reflects cofibrations, fibrations, and weak equivalences
(this is much stronger than Quillen equivalence).  If $M$ is a model
category equivalent to one of the form $\sPsh(C)_\mathscr{S}$ for some
presentation $(C,\mathscr{S})$, then we write 
\[
\thetaspenr{M} \defeq \sPsh(\Theta C)_{\mathscr{S}_\Theta}.
\]
We call $\thetaspenr{M}$ the model category of
\dfn{$\Theta$-spaces over $M$}. 

In the rest of this section, we prove the following result, which is
the the precise form of \eqref{thm:main-intro}.
\begin{thm}\label{thm:main-thm-precise}
If $(C,\mathscr{S})$ is a cartesian presentation, then $(\Theta C,
\mathscr{S}_\Theta)$ is a cartesian presentation, so 
that $\sPsh(\Theta C)_{\mathscr{S}_\Theta}$ is a cartesian model
category. 
\end{thm}

\subsection{$V[1](\mathscr{S})$-fibrant objects}

The $V[1](\mathscr{S})$-fibrant objects are precisely the injective
fibrant objects whose mapping spaces are $\mathscr{S}$-fibrant.
Explicitly, we have the following.
\begin{prop}\label{prop:v1-fib-charac}
An injectively fibrant object $X$ in $\sPsh(\Theta C)$ is
$V[1]\mathscr{S}$-fibrant if and only if for each $x_1,x_2\in X[0]$, the
object $M_X(x_1,x_2)$ is an $\mathscr{S}$-fibrant object of $\sPsh(C)$.
\end{prop}

\subsection{Proof of \eqref{thm:main-thm-precise}}

It is clear that \eqref{thm:main-thm-precise} follows from
\eqref{prop:presentations-segal-plus},
\eqref{prop:se-plus-cpt-is-cartesian}, and the 
following.
\begin{prop}\label{prop:se-v-cartesian}
If $(C,\mathscr{S})$ is a cartesian presentation, then $(\Theta C,
\Se_C\cup V[1](\mathscr{S}))$ is a cartesian presentation.
\end{prop}
In the remainder of the section we prove this proposition
\eqref{prop:se-v-cartesian}. 

In light of \eqref{prop:v1-fib-charac} and
\eqref{prop:presentations-segal-plus}, it is enough to show that if
$(C,\mathscr{S})$ is a cartesian presentation and $X$ is $\Se_C\cup
V[1](\mathscr{S})$-fibrant, and if $Y=X^{F[1](d)}$ for some $d\in \ob
C$, 
then $M_{Y}(g_0,g_1)$ is an
$\mathscr{S}$-fibrant object of $\sPsh(C)$, for all 
points $g_0,g_1$ in $X^{F[1](d)}[0]$.

Let $c$ and $d$ be objects in $C$, and consider the following
diagram in $\sPsh(\Theta C)$. 
\begin{equation}\label{eq:prod-decomp-for-v-cartesian}
\vcenter{
\xymatrix{
{V[2](Fc,Fd)} 
& {V[1](Fc\times Fd)} \ar[l]_{V\delta^{02}} \ar[r]^{V\delta^{02}}
& {V[2](Fd,Fc)}
\\
{V[2](\varnothing,Fd)} \ar[u]
& {V[1](\varnothing)} \ar[u] \ar[l]^-{V\delta^{02}}
\ar[r]_-{V\delta^{02}}
& {V[2](Fd,\varnothing)} \ar[u]
}}
\end{equation}
By \eqref{prop:product-decomposition}, taking colimits along the rows
gives the map
\[
f\colon V[1](Fd)\amalg V[1](Fd)\approx V[1](\varnothing)\times
V[1](Fd)\ra V[1](Fc)\times V[1](Fd) 
\]
induced by $\varnothing\ra Fc$.  (Recall that
$V[1](\varnothing)\approx 1\amalg 1$.)

Now $\Map_{\Theta C}(f,X)$ is isomorphic to the
map
\[
(Y\delta^0,Y\delta^1)\colon Y[1](c) \ra Y[0]\times Y[0]
\]
whose fiber over $(g_0,g_1)$ is $M_Y(g_0,g_1)$.

Applying $\Map_{\sPsh(\Theta C)}({-},X)$ to the diagram
\eqref{eq:prod-decomp-for-v-cartesian} gives
\begin{equation}\label{eq:v-cartesian-fibers-diagram}
\vcenter{
\xymatrix{
{X[2](c,d)} \ar[r] \ar[d]
& {\Map(V[1](Fc\times Fd),X)} \ar[d] 
& {X[2](d,c)} \ar[l] \ar[d]
\\
{X[0]\times X[1](d)} \ar[r]
& {X[0]\times X[0]} 
& {X[1](d)\times X[0]} \ar[l]
}
}
\end{equation}
The space
$M_{Y}(g_0,g_1)$ is the pullback of the diagram obtained by taking
fibers of each of the vertical maps of
\eqref{eq:v-cartesian-fibers-diagram}, over points 
$(x_{00}, g_1)$, $(x_{00},x_{11})$, and 
$(g_0,x_{11})$ respectively, where $x_{ij}=(X\delta^j)(g_i)$.  The
vertical maps of \eqref{eq:v-cartesian-fibers-diagram} are fibrations
of spaces, and thus the pullback of 
fibers is a homotopy 
pullback.  Thus,  
it suffices to show that the fiber of each of the vertical maps,
viewed as a functor of 
$c$, is an $\mathscr{S}$-fibrant object of $\sPsh(C)$.  

We claim that these fibers, as presheaves on $C$, are weakly equivalent
to the presheaves 
$M_X(x_{00},x_{10})$, $(M_X(x_{00},x_{11}))^{Fd}$, and
$M_X(x_{01},x_{11})$ respectively.  The objects $M_X(x_{00},x_{10})$,
$M_X(x_{00},x_{11})$, 
and $M_X(x_{01},x_{11})$ are $\mathscr{S}$-fibrant by the hypothesis
that $X$ is $V[1](\mathscr{S})$-fibrant,  Since $(C,\mathscr{S})$ is a
cartesian 
presentation, it follows that $(M_X(x_{00},x_{11}))^{Fd}$ is
$\mathscr{S}$-fibrant.  Thus, we complete the proof of the proposition
by proving this claim. 

The left-hand vertical arrow of \eqref{eq:v-cartesian-fibers-diagram} factors
\[
X[2](c,d) \xra{(X\delta^{01},X\delta^{12})} X[1](c)\times_{X[0]}
X[1](d) \xra{((X\delta^0)\pi_1, \pi_2)} X[0]\times 
X[1](d).
\]
The first map is a weak equivalence since $X$ is $\Se_C$-local, so
it suffices to examine the 
fibers of the second map over $(x_0, g_1)$.  It is
straightforward to check that this fiber is isomorphic to
$M_X(x_{00}, x_{10})(c)$.

The right-hand vertical arrow of \eqref{eq:v-cartesian-fibers-diagram}
is analysed similarly, so that its 
fibers are weakly equivalent to $M_X(x_{01},x_{11})(c)$.

For the middle vertical arrow of \eqref{eq:v-cartesian-fibers-diagram},
\eqref{lemma:mapping-space-identification} allows us to identify 
the fiber over $(x_{00}, x_{11})$ with 
\[
\Map_{\sPsh(C)}(Fc\times Fd, M_X(x_{00}, x_{11}))
\approx (M_X(x_{00}, x_{11}))^{Fd}(c).
\]

\section{Groupoid objects}
\label{sec:groupoid-objects}

Let $\Gpd_C$ be the set consisting of the morphism
\[
T_\#q\colon T_\#F[1]\ra T_\#E.
\]
We say that a Segal object is a \dfn{Segal groupoid} if it is
$\Gpd_C$-local; likewise, a complete Segal object is called a
\dfn{complete Segal groupoid} if it is $\Gpd_C$-local.

\begin{lemma}\label{lemma:gpd-local-charac}
If $X$ is a Segal object in $\sPsh(\Theta C)$, then $X$ is
$\Gpd_C$-local if and only if $X^{T_\#i}\colon X^{T_\#E}\ra
X^{T_\#F[1]}$ is a levelwise weak equivalence in $\sPsh(\Theta C)$.
\end{lemma}
\begin{proof}
The if part is immediate.  To prove the only if part, note that for
any Segal object $Y$, the map $Y^{T_\#E}\ra Y^{T_\#F[1]}$ is fully
faithful by \eqref{prop:equiv-to-mor-fully-faithful}.
If $X$
is $\Se_C\cup \Gpd_C$-fibrant, then $X^{T_\#i}[0]\colon
X^{T_\#E}[0]\ra X^{T_\#F[1]}[0]$ is a weak equivalence of spaces, and
it follows that $X^{T_\#i}$ must be a levelwise
weak equivalence in $\sPsh(\Theta C)$.  
\end{proof}

\begin{prop}
The presentations $(\Theta C,\Se_C\cup \Gpd_C)$ and $(\Theta
C,\Se_C\cup \Cpt_C\cup 
\Gpd_C)$ are cartesian presentations.  If $(C,\mathscr{S})$ is a cartesian
presentation, then $(\Theta C,\Se_C\cup \Gpd_C\cup V[1]\mathscr{S})$
and $(\Theta C, \Se_C\cup \Cpt_C\cup \Gpd_C \cup V[1]\mathscr{S})$ are
cartesian presentations.
\end{prop}
\begin{proof}
We only need to show that $(\Theta C, \Se_C\cup \Gpd_C)$ is a
cartesian presentation; the other results follow using
\eqref{prop:se-plus-cpt-is-cartesian} and \eqref{thm:main-thm-precise}. 

To show
that $(\Theta C, \Se_C\cup \Gpd_C)$ is  a 
cartesian presentation, we need to show
\eqref{prop:presentations-segal-plus} that if $X$ is $\Se_C\cup
\Gpd_C$-fibrant, then $Y=X^{F[1](c)}$ is
$\Gpd_C$-local for all $c\in \ob C$.
The map $\Map_{\Theta C}(T_\#i,Y)\colon \Map_{\Theta C}(T_\#E,Y)\ra
\Map_{\Theta C}(T_\#F[1],Y)$ is isomorphic to $(X^{T_\#i})[1](c)\colon
(X^{T_\#E})[1](c) \ra (X^{T_\#F[1]})[1](c)$, which is a weak
equivalence by \eqref{lemma:gpd-local-charac}.
\end{proof}

Given a presentation $(C,\mathscr{S})$ with
$M=\sPsh(C)^\inj_\mathscr{S}$, let
\[
\Theta_{\Gpd}(C,\mathscr{S}) = (\Theta C, \mathscr{S}_\Theta\cup
\Gpd_C)
\]
and
\[
\thetagpdenr{M} \defeq \sPsh(\Theta C)^\inj_{\mathscr{S}_\Theta\cup \Gpd_C}.
\]

\section{Alternate characterization of complete Segal objects}
\label{sec:alternate-cs}

In the section we consider a characterization of
the ``completeness'' property in the definition of a complete Segal
space, which is a bit more elementary than the one given in
\cite{rezk-ho-theory-of-ho-theory}.  The results of this section are
not needed elsewhere in this paper.

Let $E$ in $\sPsh(\Delta)$ be the
``discrete nerve'' of the groupoid with two uniquely isomorphic
objects $x$ and $y$, let $p\colon E\ra F[0]$ denote the projection, and let
$i\colon F[1]\ra E$ denote the map which picks out the morphism from
$x$ to $y$.  In \cite{rezk-ho-theory-of-ho-theory}*{Prop.\ 6.2} it is shown
that if $X$ is a Segal space, then $\Map(i,X)\colon \Map(E,X)\ra
\Map(F[1],X)\approx X[1]$ factors through a weak equivalence
$\Map(E,X)\ra X^\hequiv$.  From this, we see that a Segal space $X$ is
complete if and only if $\Map(p,X)$ is a weak equivalence
\cite{rezk-ho-theory-of-ho-theory}*{Prop.\ 6.4}.

The proof of \cite{rezk-ho-theory-of-ho-theory}*{Prop.\ 6.2} is long
and technical.  Also, the result is not entirely satisfying, because
$E$ is an ``infinite dimensional'' object, in the sense that as a
simplicial space it is constructed from infinitely many cells, which
appear in all dimensions (see
\cite{rezk-ho-theory-of-ho-theory}*{\S11}).  It is possible to replace
$E$ with the finite subobject $E^{(k)}$ for $k\geq3$ (see
\cite{rezk-ho-theory-of-ho-theory}*{Prop.\ 11.1}), but this is also
not very satisfying.

Here we prove a variant of \cite{rezk-ho-theory-of-ho-theory}*{Prop.\
  6.2} where $E$ is replaced by an object $Z$, which is a finite cell
object. 
The idea is based on the following observation:
in a category
enriched over spaces, the homotopy equivalences $g\colon X\ra Y$ are
precisely those morphisms for which there exist morphisms $f,h\colon
Y\ra X$ and homotopies $\alpha\colon gf\sim 1_Y$ and $\beta\colon
hg\sim 1_X$, and that for a given homotopy equivalence $g$ the
``moduli space'' of such data $(f,h,\alpha,\beta)$ is weakly
contractible.

Define an object $Z$ in $\sPsh(\Delta)$ to be the colimit of the
diagram
\[
F[3] \xla{(\delta^{02},\delta^{13})} F[1]\amalg F[1]
\xra{\delta^{00}\amalg \delta^{00}} F[0]\amalg F[0].
\]
Let $p\colon Z\ra F[0]$ be the evident projection map, and let
$i\colon F[1]\ra Z$ be the composite of $\delta^{12}\colon F[1]\ra 
F[3]$ with the quotient map $F[3]\ra Z$.  

\begin{prop}
Let $X$ be a Segal space (i.e., an $\Se_1$-fibrant object of
$\sPsh(\Delta)$).  The map $\Map(Z,X)\ra \Map(F[1],X)$ factors through
$X^\hequiv \subseteq X[1]$, and induces a weak equivalence
$\Map(Z,X)\ra X^\hequiv$ of spaces.

Thus, a Segal space $X$ is a complete Segal space if and only if the
square
\[
\xymatrix{
{X[0]} \ar[rr]^{X\delta^{0000}} \ar[d]_{(X\delta^0,X\delta^0)}
&& {X[3]} \ar[d]^{(X\delta^{02},X\delta^{13})}
\\
{X[0]\times X[0]} \ar[rr]_{X\delta^{00}\times X\delta^{00}}
&& {X[1]\times X[1]}
}\]
is a homotopy pullback.
\end{prop}
\begin{proof}
Consider the following commutative diagram 
\[\xymatrix{
{P} \ar[rr]^d \ar@{>->}[d]_b
&& {T} \ar[rr]^e \ar@{>->}[d]_c 
&& {X^\hequiv} \ar@{>->}[d]^{j}
\\
{Q} \ar[rr]^{a}  \ar[d]
&& {X[3]} \ar[rr]^{X\delta^{12}}  \ar[d]_{(X\delta^{02},X\delta^{13})}
&& {X[1]} \ar[d]^{(X\delta^1, X\delta^0)}
\\
{X[0]\times X[0]} \ar[rr]_{X\delta^{00}\times X\delta^{00}}
&& {X[1]\times X[1]} \ar[rr]_{X\delta^1\times X\delta^0}
&& {X[0]\times X[0]}
}\]
Here, the objects $Q$, $T$, and $P$ are defined to be the 
pullbacks of the lower left, upper right, and upper left squares
respectively; each of these squares is a homotopy pullback of spaces,
since $(X\delta^{02},X\delta^{13})$ and $j$ are fibrations.
(The lower right square is in general \emph{not} a 
pullback or a homotopy pullback.)
The maps $b$, $c$, and $j$ are homotopy monomorphisms.
Observe that $Q\approx \Map(Z,X)$, and so we want to prove that
$(X\delta^{12})a$ factors through a weak equivalence $k\colon Q\ra
X^\hequiv$. 

The result will follow by showing (i) that the horizontal map
$(X\delta^{12})a\colon Q\ra
X_1$ factors through the inclusion $j\colon X^\hequiv\ra X_1$ by a map
$k\colon Q\ra X^\hequiv$ (and
thus $b\colon P\ra Q$ is a 
weak equivalence), and (ii) that the right hand rectangle is a
homotopy pullback, i.e., that $T\approx \holim(X_h \ra X_0\times X_0
\la X_1\times X_1)$.  Condition (ii) implies that $ed\colon P\ra
X^\hequiv$ is a weak 
equivalence, since it is a homotopy pullback of the identity map of
$X_0\times X_0$.  Since $fb=ed$, it follows that $k$ is a weak
equivalence, as desired.

The proof of (i) is straightforward.  If $H$ is a point in $Q$, let
$g\defeq ((X\delta^{12})a)(H)$ in $X[1]$.  Then by construction the
class $[g]$ in the homotopy category $hX$ admits both a left and a
right inverse, and thus $g$ is a point of $X^\hequiv$.  (See the
discussion in \cite{rezk-ho-theory-of-ho-theory}*{\S5.5}.)

To prove (ii), let 
\[
T'=\lim( X^\hequiv\xra{(X\delta^1,X\delta^0)j} X[0]\times X[0]
\xla{X\delta^1\times X\delta^0} X[1]\times
X[1]).
\]
Since $X\delta^1$ and $X\delta^0$ are fibrations, this is a homotopy
pullback.  We need to show that $t\colon T\ra T'$ is a weak equivalence.
Let
$\pi'\colon X[1]\times X[1] \ra  X[0]\times X[0] \times
X[0] \times X[0],$ be the map defined by
\[
\pi'(u,v) \defeq ((X\delta^0)u, (X\delta^0)v, (X\delta^1)u, (X\delta^1)v).
\]
Let $\pi\colon T'\ra (X[0])^4$ be the composite of $\pi'$ with the
tautological map $T'\ra X[1]\times X[1]$.  Note that both $\pi$ and
$\pi t$ are fibrations of spaces.

Let  $\underline{x}=(x_0,x_1,x_2,x_3)$ be a tuple of points in $X[0]$.
The fiber of $\pi$ over $\underline{x}$ is the 
space
\[
T_{\underline{x}}'\defeq M_X(x_0,x_2)\times M_x^\hequiv(x_1,x_2)\times
M_X(x_1,x_3). 
\]
The fiber of $\pi t$ over $\underline{x}$ is the limit 
\[
T_{\underline{x}}\defeq \lim\bigl( M_X(x_0,x_1,x_2,x_3) \ra
M_X(x_1,x_2) \leftarrowtail M_X^\hequiv(x_1,x_2) \bigr).
\]
To prove the proposition, we need to show that for all
$\underline{x}$, the map $t_{\underline{x}}\colon T_{\underline{x}}\ra
T_{\underline{x}}'$ induced 
by $t$ is a weak equivalence. 

Given a point $f$ in $M_X(x_0,x_1)$, we write $M_X(x_0,x_1)_f$ for the
path component of $M_X(x_0,x_1)$ containing $f$.  Given 
sequence of points $f_i$ in $M_X(x_{i-1},x_i)$, we write
$M_X(x_0,\dots,x_n)_{f_1,\dots,f_n}$ for the path component of
$M_X(x_0,\dots,x_n)$ which projects to $M_X(x_0,x_1)_{f_1}\times
\cdots \times M_X(x_{n-1},x_n)_{f_n}$ under the Segal map.
We claim that if $f\in M_X(x_0,x_1)$, $g\in M_X^\hequiv(x_1,x_2)$, and
$h\in M_X(x_2,x_3)$, then the maps
\[
\zeta\colon M_X(x_0,x_1,x_2,x_3)_{f,g,h}
\ra M_X(x_0,x_1,x_2)_{f,g}\times M_X(x_1,x_3)_{h\circ g}
\]
and
\[
\eta\colon M_X(x_0,x_1,x_2)_{f,g}\times M_X(x_1,x_3)_{h\circ g} \ra
M_X(x_0,x_2)_{g\circ f}\times M_X(x_1,x_2)_g \times
M_X(x_1,x_3)_{h\circ g}
\]
are weak equivalences.  This is a straightforward calculation, using
the ideas of \cite{rezk-ho-theory-of-ho-theory}*{Prop.\ 11.6}.
The map $t_{\underline{x}}$ is the disjoint union of maps $\eta\zeta$
over the appropriate path components, and thus the proposition is proved.
\end{proof}

\section{$(n+k,n)$-$\Theta$-spaces}
\label{sec:theta-spaces}

In this section, we do three things.  First, we make precise the
``informal description'' of $(n+k,n)$-$\Theta$-spaces given in
\S\ref{subsec:informal-description}.  Next, we identify the
``discrete'' $(\infty,n)$-$\Theta$-spaces
\eqref{prop:rigid-n-cat-is-discrete-theta-space}.  Finally, we show
that
``groupoids'' in $(n+k,n)$-$\Theta$-spaces are essentially the same as
$(n+k)$-truncated spaces \eqref{prop:quillen-equiv-groupoids-spaces},
thus proving the ``homotopy hypothesis'' for 
these models.  

\subsection{Functor associated to a presheaf on $\Theta_n$}

Given an object $X$ of $\sPsh(\Theta_n)$, let $\bar{X}\colon
\sPsh(\Theta_n)^\op\ra \Sp$ denote the functor defined by
\[
\bar{X}(K) \defeq \Map_{\Theta_n}(K,X).
\]
The construction $X\mapsto \bar{X}$ is the Yoneda embedding of
$\sPsh(\Theta_n)$ into the category of $\Sp$-enriched functors
$\sPsh(\Theta_n)^\op\ra \Sp$.  The object $X$ is recovered from the
functor $\bar{X}$ by the formula
$X(\theta) \approx \bar{X}(F\theta)$.

\subsection{The discrete nerve}

Given a strict $n$-category $C$, we define the \dfn{discrete nerve} of
$C$ to be the presheaf of sets $\dnerve C$ on $\Theta_n$ defined by
\[
(\dnerve C)(\theta) = \stcat{n}(\tau_n \theta, C).
\]
Since we can regard presheaves of sets as a full subcategory of
discrete presheaves of simplicial sets, we will regard $\dnerve$ as a
functor 
$\dnerve\colon \stcat{n}\ra \sPsh(\Theta_n)$.  This functor is fully
faithful.  Finally, note that there is a natural isomorphism
$F\approx \dnerve \tau$,
where $\tau_n\colon \Theta_n\ra \stcat{n}$ is the inclusion functor of
\eqref{cor:tau-n-fully-faithful}, and $F\colon \Theta_n\ra
\sPsh(\Theta_n)$ is the Yoneda 
embedding of $\Theta_n$.

\subsection{The suspension and inclusion functors}

For all $n\geq1$ there is a \dfn{suspension} functor
\[
\sigma\colon \Theta_{n-1}\ra \Theta_n
\]
defined on objects by $\sigma(\theta) \defeq [1](\theta)$.  Composing
suspension functors gives functors $\sigma^k\colon \Theta_{n-k}\ra
\Theta_n$ for $0\leq k\leq n$.

For all $n\geq1$ there is an \dfn{inclusion} functor
\[
\tau\colon \Theta_{n-1}\ra \Theta_n
\]
which is the restriction of the standard inclusion $\stcat{n-1}\ra
\stcat{n}$ to $\Theta_{n-1}$.  Composing inclusion functors gives
functors $\tau^k\colon \Theta_{n-k}\ra \Theta_n$ for $0\leq k\leq n$.

\subsection{The category $\thetaspk{n}{k}$}

For $0\leq n<\infty$, let $\mathscr{T}_{n,\infty}$ be the set of morphisms in
$\sPsh(\Theta_n)$ defined by
\begin{align*}
\mathscr{T}_{0,\infty} &= \varnothing,
\\
\mathscr{T}_{n,\infty} &= \Se_{\Theta_{n-1}}\cup
\Cpt_{\Theta_{n-1}} \cup 
V[1](\mathscr{T}_{n-1,\infty})\text{ for $n>0$.}
\end{align*}
If also given $-2\leq k<\infty$, let $\mathscr{T}_{n,k}$ be the set of
morphisms in $\sPsh(\Theta_n)$ defined by 
\begin{align*}
\mathscr{T}_{0,k} &= \{\partial \Delta^{k+2}\ra \Delta^{k+2}\},
\\
\mathscr{T}_{n,k} &= \Se_{\Theta_{n-1}}\cup \Cpt_{\Theta_{n-1}} \cup
V[1](\mathscr{T}_{n-1,k})\text{ for $n>0$.}
\end{align*}
In the notation of \S\ref{sec:presentation-theta-css},
$\mathscr{T}_{n,k}=(\mathscr{T}_{n-1,k})_\Theta$ for $n>0$.
\begin{prop}
For all $0\leq n<\infty$ and $-2\leq k\leq \infty$, the presentation
$(\Theta_n, \mathscr{T}_{n,k})$ is cartesian.
\end{prop}
\begin{proof}
Immediate from \eqref{thm:main-thm-precise}.
\end{proof}

Let $\thetaspk{n}{k}\defeq \sPsh(\Theta_n)^\inj_{\mathscr{T}_{n,k}}$;
we call this the \dfn{$(n+k,n)$-$\Theta$-space model category}.  We
show that the fibrant objects of $\thetaspk{n}{k}$ are precisely the
$(n+k,n)$-$\Theta$-spaces described in 
\S\ref{subsec:informal-description}.  

\subsection{The structure of the sets $\mathscr{T}_{n,k}$}
\label{subsec:structure-of-t-nk}

For $n\geq0$ and $-2\leq k<\infty$, we have
\[
\mathscr{T}_{n,\infty} = \mathscr{T}^{\Se}_n \cup \mathscr{T}^{\Cpt}_n
\]
and
\[
\mathscr{T}_{n,k} = \mathscr{T}_{n,\infty}\cup \{(V[1])^{n}(\partial
\Delta^{k+2}\ra \Delta^{k+2})\},
\]
where
\[
\mathscr{T}^{\Se}_n = \set{(V[1])^k(
  \se^{\theta_1,\dots,\theta_r})}{0\leq k<n,\; r\geq2,\;
  \theta_1,\dots,\theta_r\in \ob \Theta_{n-k}}
\]
and 
\[
\mathscr{T}^{\Cpt}_n = \set{(V[1])^k(T_\# p)}{0\leq k<n}.
\]

\begin{prop}
For $\theta_1,\dots,\theta_r\in \ob\Theta_{n-k}$, the map
$(V[1])^k(\se^{\theta_1,\dots,\theta_r})$ is isomorphic to the map 
\[
\colim\bigl( F\sigma^k [1](\theta_1)\la F\sigma^k[0]\ra \cdots \la
F\sigma^k[0] \ra F\sigma^k[1](\theta_r)\bigr) \ra
F\sigma^k[r](\theta_1,\dots, \theta_r),
\]
induced by applying $F\sigma^k$ to the maps $\delta^{i-1,i}\colon
[1](\theta_i)\ra [r](\theta_1,\dots,\theta_r)$.
\end{prop}
\begin{proof}
Immediate using \eqref{lemma:v-of-f-theta} and
\eqref{prop:intertwining-functor-colimit-properties}.  
\end{proof}

\subsection{The objects $O_k$ and $\partial O_k$}

Fix $n\geq0$.  We write $O_k$ for the discrete nerve of the
free-standing $k$-cell in $\stcat{n}$.  It follows that $O_k \approx
F\sigma^k[0] \approx F[1]([1](\cdots [1]([0])))$, where
$\sigma^k\colon \Theta_{n-k}\ra \Theta_n$.   Note that our usage of
$O_k$ here is slightly different than that described in the
introduction, where 
$O_k$ was used to mean the object of $\Theta_n$, rather than the
object of $\sPsh(\Theta_n)$.

If $k>0$, then the free-standing $k$-cell in $\stcat{n}$ is a
$k$-morphism between two parallel $(k-1)$-cells.  Let $s_k,t_k\colon
O_{k-1}\ra O_k$ denote the map between discrete nerves induced by the
inclusion of the two parallel $(k-1)$-cells.  Equivalently, $s_k$ and
$t_k$ are the maps obtained by applying $\sigma^{k-1}$ to the maps
$\delta^0,\delta^1\colon [0]\ra [1]$ of $\Theta_{n-k}$.

Let $\partial O_k$  denote the maximal proper subobject of $O_k$; that
is, $\partial O_k\subset O_k$ is the largest sub-$\Theta_n$-presheaf of $O_k$ 
which does not contain the ``tautological section'' $\iota \in
O_k(\sigma^k[0])$.  Let $e_k\colon \partial O_k\ra O_k$ denote the
inclusion.  
\begin{prop}
For $1\leq k\leq n$, the map
\[
\colim\bigl( O_{k-1} \xla{e_{k-1}} \partial O_{k-2} \xra{e_{k-1}}
O_{k-1}\bigr) \ra \partial O_k
\]
defined by $s_k,t_k\colon O_{k-1}\ra O_k$
is an isomorphism in $\sPsh(\Theta_n)$.
\end{prop}
By abuse of notation, we write $s_k,t_k\colon O_{k-1}\ra \partial O_k$
for the inclusion of the two copies of $O_{k-1}$.

It is clear that $\partial O_k$ is isomorphic to the discrete nerve of the
``free-standing pair of parallel $(k-1)$-cells''.  Observe that
$\partial O_0=\varnothing$.

\begin{lemma}\label{lemma:v-of-f-theta}
For $\theta\in \ob \Theta_{n-1}$, the object $V[1](F\theta)\approx
F([1](\theta))\approx F\sigma(\theta)$ as objects of $\sPsh(\Theta_{n})$.
\end{lemma}
\begin{proof}
A straightforward calculation.
\end{proof}

\begin{prop}\label{prop:v-of-cell-and-boundary}
For $1\leq k<n$, the functor $V[1]\colon \sPsh(\Theta_{n-1})\ra
\sPsh(\Theta_n)$ carries the diagram
\[
s_k,t_k\colon O_{k-1}\rightrightarrows O_k \la \partial O_k\noloc e_k
\]
up to isomorphism to the diagram
\[
s_{k+1},t_{k+1} \colon O_k \rightrightarrows O_{k+1} \la \partial
O_{k+1}\noloc e_{k+1}.
\]
Furthermore, $V[1](\varnothing)=V[1](\partial O_0)\approx \partial O_1$.
\end{prop}
\begin{proof}
Again, a straightforward calculation using \eqref{lemma:v-of-f-theta}
and  \eqref{prop:intertwining-functor-colimit-properties}.
\end{proof}

\subsection{Mapping objects between pairs of parallel $(k-1)$-cells}

Let $X$ be a $\mathscr{T}^\Se_n$-fibrant object in $\sPsh(\Theta_n)$.
We call the space $\bar{X}(O_k) = \Map_{\Theta_n}(O_k, X)$ the
\dfn{moduli space of $k$-cells} of $X$.  We call the space
$\bar{X}(\partial O_k) = \Map_{\Theta_n}(\partial O_k, X)$ the
\dfn{moduli space of pairs of parallel $(k-1)$-cells}.
(These are the spaces denoted $X(O_k)$ and $X(\partial O_k)$ in the
introduction.) 

Observe that the maps $s_k,t_k\colon O_{k-1}\ra \partial O_k$
determine an isomorphism
\[
\bar{X}(\partial O_k)\xra{\sim}
\bar{X}(O_{k-1})\times_{\bar{X}(\partial O_{k-1})} \bar{X}(O_{k-1}).
\]
In particular, $\bar{X}(\partial O_k) \ra
\bar{X}(O_{k-1})\times \bar{X}(O_{k-1})$ is a monomorphism, so that a point of
$\bar{X}(\partial O_k)$ can be named by a suitable pair of points in
$\bar{X}(O_{k-1})$. 

Suppose $1\leq k\leq n$, and suppose given $(f_0,f_1)\in
\bar{X}(\partial O_k)$.  We write $\umap_X(f_0,f_1)$ for the object of
$\sPsh(\Theta_{n-k})$ defined by
\[
\umap_X(f_0,f_1)(\theta) \defeq \lim\bigl( \bar{X}(V[1]^k(F\theta))) \ra
\bar{X}(V[1]^k(\varnothing)) \approx 
\bar{X}(\partial O_k) \la \{(f_0,f_1)\} \bigr).
\]
Observe that these objects can be obtained by iterating the mapping
object construction of \eqref{subsec:mapping-objects}.  In particular,
if $(x_0,x_1)\in X[0]\times X[0]\approx \bar{X}(\partial O_1)$, then
$\umap_X(x_0,x_1)\approx M_X(x_0,x_1)$ as objects of $\sPsh(\Theta_{n-1})$.

\begin{lemma}\label{lemma:se-fibrant-mapping-gadget}
If $X$ is a $\mathscr{T}^\Se_n$-fibrant object of $\sPsh(\Theta_n)$,
then $\umap_X(f_0,f_1)$ is a $\mathscr{T}^\Se_{n-k}$-fibrant object of
$\sPsh(\Theta_{n-k})$.  
\end{lemma}
\begin{proof}
Immediate from the fact that $\mathscr{T}^\Se_n\supseteq
V[1]^k(\mathscr{T}^\Se_{n-k})$. 
\end{proof}

\subsection{The moduli space $X(O_k)^\hequiv$ of
  $k$-equivalences} 

Let $X$ be a $\mathscr{T}^\Se_n$-fibrant object of $\sPsh(\Theta_n)$,
and suppose $1\leq k\leq n$.
Given a $k$-cell in $X$, i.e., a point $g$ in $\bar{X}(O_k)$, let 
\[
b_0= (\bar{X}s_k)(g)\quad \text{and}\quad b_1=(\bar{X}t_k)(g)
\]
 be
the ``source'' and 
``target'' $(k-1)$-cells of $g$, and let 
\[
a_0 =
(\bar{X}s_{k-1})b_0=(\bar{X}s_{k-1})b_1\quad \text{and}\quad a_1 =
(\bar{X}t_{k-1})b_0=(\bar{X}t_{k-1})b_1
\]
 be the ``source'' and
``target'' $(k-2)$-cells of $b_0$ and $b_1$.    Let
$Y=\umap_X(a_0,a_1)$ as an object of 
$\sPsh(\Theta_{n-k+1})$; by \eqref{lemma:se-fibrant-mapping-gadget}
the 
presheaf $Y$ is $\mathscr{T}^\Se_{n-k+1}$-fibrant.  Then $b_0$ and
$b_1$ are ``objects'' of $Y$, 
(that is, points in $Y[0]=\bar{Y}(O_0)$), and $g$ is a ``$1$-cell'' of
$Y$, (that is, a point of $\bar{Y}(O_1)$).  Recall
(\S\ref{subsec:homotopy-category-segal-object}) that $g$ thus
represents an element $[g]$ of the homotopy category $hY$ of $Y$.

Say that a $k$-cell $g$ of $X$ is a \dfn{$k$-equivalence} if it
represents an isomorphism in the homotopy category $hY$ of
$Y=\umap_X(a_0,a_1)$.  Let $\bar{X}(O_k)^\hequiv\subseteq
\bar{X}(O_k)$ denote the union of path components of $\bar{X}(O_k)$
which contain $k$-equivalences.

\subsection{Characterization of $\mathscr{T}^\Se_n\cup
  \mathscr{T}^\Cpt_n$-fibrant objects} 

Recall the map $i\colon F[1]\ra E$ of \S\ref{subsec:the-set-cpt}.
\begin{prop}
For all $1\leq k\leq n$, the map 
\[
\bar{X}(V[1]^{k-1}(T_\#i))\colon
\bar{X}(V[1]^{k-1}(T_\#E))\ra \bar{X}(V[1]^{k-1}(T_\#F[1]))\approx 
\bar{X}(O_k)
\]
factors through the subspace $\bar{X}(O_k)^\hequiv \subseteq
\bar{X}(O_k)$ and induces a weak equivalence
$\bar{X}(V[1]^{k-1}(T_\#E))\ra \bar{X}(O_k)^\hequiv$.
\end{prop}
\begin{proof}
Let $(a_0,a_1)$ be a point in $\bar{X}(\partial O_{k-1})$ and let
$Y=\umap_X(a_0,a_1)$.   Since
$Y$ is $\mathscr{T}^\Se_{n-k}$-fibrant, it is in
particular $\Se_{\Theta_{n-k-1}}$-fibrant, and thus the map
\[
\Map_{\Theta_{n-k}}(T_\#i, Y)\colon \Map_{\Theta_{n-k}}(T_\# E, Y) \ra
\Map_{\Theta_{n-k}}(T_\# F[1], Y) \approx T^*Y[1]
\]
factors through $Y^\hequiv\subseteq T^*Y[1]$, and induces a weak
equivalence $\Map(T_\# E,Y)\ra Y^\hequiv$ of spaces
\eqref{prop:e-represents-equiv-segal-object}.  

Now consider the diagram
\[
\xymatrix{
{\bar{X}(V[1]^{k-1}(T_\# E))} \ar[rr]^{\bar{X}(V[1]^{k-1}(T_\# i))}
\ar[rd]
&& {\bar{X}(V[1]^{k-1}(T_\# F[1]))} \ar[dl]
\\
& {\bar{X}(V[1]^{k-1}(\varnothing))}
}\]
Over $(a_0,a_1)\in \bar{X}(\partial O_{k-1})\approx
\bar{X}(V[1]^{k-1}(\varnothing))$, the map induced by
$\bar{X}(V[1]^{k-1}(T_\#i))$ on fibers is isomorphic to the map
$\Map_{\Theta_{n-k}}(T_\# i, Y)$, and the result follows.
\end{proof}

\begin{cor}\label{cor:cpt-fibrant-charac}
Let $X$ be a $\mathscr{T}^\Se_n$-fibrant object of $\sPsh(\Theta_n)$.
Then $X$ is $\mathscr{T}^\Cpt_n$-fibrant if and only if the maps
$\bar{X}(i_k)\colon \bar{X}(O_{k-1})\ra \bar{X}(O_k)^\hequiv$ are weak
equivalences for $1\leq k\leq n$.
\end{cor}
\begin{proof}
Immediate from the structure of $\mathscr{T}^\Cpt_n$
\eqref{subsec:structure-of-t-nk}. 
\end{proof}

Thus, the $\mathscr{T}^\Se_n\cup \mathscr{T}^\Cpt_n$-fibrant objects
of $\sPsh(\Theta_n)$ are precisely the $(\infty,n)$-$\Theta$-spaces.

We record the following.
\begin{prop}\label{prop:umap-theta-space}
If $X$ is a $(\infty,n)$-$\Theta$-space, and $(f_0,f_1)$ in
$\bar{X}(\partial O_k)$ is a pair of parallel $(k-1)$-cells of $X$,
then $\umap_X(f_0,f_1)$ is an $(\infty,n-k)$-$\Theta$-space.
\end{prop}

\subsection{Characterization of $k$-truncated objects}

Let $X$ be an $(\infty,n)$-$\Theta$-space (i.e., a
$\mathscr{T}_{n,\infty}=\mathscr{T}^\Se_n\cup
\mathscr{T}^\Cpt_n$-fibrant object in 
$\sPsh(\Theta_n)$).  
Let $(f_0,f_1)$ be a point in $\bar{X}(\partial O_n)$.  Then
$\umap_X(f_0,f_1)$ is an object of $\sPsh(\Theta_0)\approx \Sp$.
Furthermore, if $K$ is a space, the fiber of $\bar{X}(V[1]^n(K))\ra
\bar{X}(V[1]^n(\varnothing)) \approx \bar{X}(\partial O_n)$ over
$(f_0,f_1)$ is
naturally isomorphic to $\Map(K, \umap_X(f_0,f_1))$.
\begin{prop}
A $\mathscr{T}_{n,\infty}$-fibrant object $X$ of $\sPsh(\Theta_n)$ is
$\mathscr{T}_{n,k}$-fibrant if and only if for all $(f_0,f_1)$ in
$\bar{X}(\partial O_n)$, the space $\umap_X(f_0,f_1)$ is $k$-truncated.
\end{prop}
\begin{proof}
On fibers over $(f_0,f_1)$, the map $\bar{X}(V[1]^n(\Delta^{k+2})) \ra
\bar{X}(V[1]^n(\partial \Delta^{k+2}))$ induces the map
$\Map(\Delta^{k+2}, \umap_X(f_0,f_1))\ra \Map(\partial \Delta^{k+2},
\umap_X(f_0,f_1))$ of spaces.
\end{proof}

\subsection{Rigid $n$-categories}

The following proposition characterizes the discrete
$\mathscr{T}^\Se_n$-fibrant objects of $\sPsh(\Theta_n)$.
\begin{prop}
The functor $\dnerve$ induces an equivalence between $\stcat{n}$ and
the full subcategory of discrete $\mathscr{T}_n^\Se$-fibrant objects
of $\sPsh(\Theta_n)$.
\end{prop}
\begin{proof}
A discrete presheaf $X$ is $\mathscr{S}$-fibrant if and only if
$\Map(s,X)\colon \Map(S',X)\ra \Map(S,X)$ is an isomorphism for all
$s\colon S\ra S'$ in $\mathscr{S}$.  It is clear that if
$\mathscr{S}=\mathscr{T}^\Se_n$, then this condition amounts to
requiring that $X$ be in the essential image of $\dnerve$.
\end{proof}

Let $C$ be a strict $n$-category.  We define the following notions for
cells in $C$, by downwards induction.
\begin{enumerate}
\item Let $g\colon x\ra y$ be a $k$-morphism in $C$.  If $1\leq k<n$, we
  say $g$ is a 
  \dfn{$k$-equivalence} if there exist $k$-cells $f,h\colon y\ra x$ in
  $C$ such that $gf\sim 1_y$ and $hg\sim 1_x$.  If $k=n$, we say $g$
  is a $k$-equivalence if it is a $k$-isomorphism.
\item  Let $f,g\colon x\ra y$ be two parallel $k$-cells in $C$.  If $0\leq
  k<n$, we say that   $f$ and $g$ are \dfn{equivalent}, and write
  $f\sim g$, if there exists  a $(k+1)$-equivalence $h\colon f\ra g$.
  If $k=n$, we say that $f$ and $g$ are equivalent if there are equal.
\end{enumerate}

\begin{prop}\label{prop:rigid-charac}
Let $C$ be a strict $n$-category.  The following are equivalent.
\begin{enumerate}
\item For all $1\leq k\leq n$, every $k$-equivalence is an identity
  $k$-morphism. 
\item For all $1\leq k\leq n$, every $k$-isomorphism is an identity
  $k$-morphism. 
\end{enumerate}
\end{prop}
\begin{proof}
  It is clear that (1) implies (2).  To show that (2) implies (1), use
  downward induction on $k$.
\end{proof}

By a \dfn{rigid $n$-category}, we mean a strict $n$-category $C$
satisfying either of the equivalent conditions of
\eqref{prop:rigid-charac}. 
\begin{prop}\label{prop:rigid-n-cat-is-discrete-theta-space}
Let $C$ be a strict $n$-category.  The discrete nerve $\dnerve C$ is an
$(\infty,n)$-$\Theta$-space (i.e., is $\mathscr{T}_n^{\Se}\cup
\mathscr{T}_n^{\Cpt}$-fibrant) if and only if $C$ is a rigid $n$-category.
\end{prop}
\begin{proof}
Let $C$ be a strict $n$-category.
By \eqref{prop:rigid-charac} $\dnerve C$ is
$\mathscr{T}_n^\Se$-fibrant.  It will also be
$\mathscr{T}_n^\Cpt$-fibrant if and only if $\bar{X}(O_k)^\hequiv \ra
\bar{X}(O_k)$ 
is an isomorphism for all $1\leq k\leq n$, and the result follows from
the observation that $\bar{X}(O_k)^\hequiv$ corresponds precisely to the set
of $k$-isomorphisms in $C$.
\end{proof}

\subsection{Groupoids and the homotopy hypothesis}
\label{subsec:homotopy-hypothesis}

For $n\geq0$, let 
\begin{align*}
\mathscr{T}^\Gpd_0 &\defeq \varnothing,\\
\mathscr{T}^\Gpd_n &\defeq \Gpd_{\Theta_{n-1}} \cup
V[1](\mathscr{T}^\Gpd_{n-1}),
\end{align*}
using the definition of $\Gpd_C$ of \S\ref{sec:groupoid-objects}.
Explicitly, we have
\[
\mathscr{T}^\Gpd_n \approx \set{V[1]^k(T_\#q)}{0\leq k< n},
\]
where $T_\#q\colon T_\# F[1]\ra T_\#E$ is as in
\S\ref{sec:groupoid-objects}. 

Let $\thetagpdk{n}{k} \defeq
\sPsh(\Theta_n)^\inj_{\mathscr{T}_{n,k}\cup \mathscr{T}^\Gpd_n}$; we
call this the \dfn{$(n+k,n)$-$\Theta$-groupoid model category}.  
The fibrant objects of $\thetagpdk{n}{k}$ are called
\dfn{$(n+k,n)$-$\Theta$-groupoids}; they form a full subcategory of
the category of $(n+k,n)$-$\Theta$-spaces.

\begin{prop}
Let $n\geq0$.  
Let $X$ be a $(\infty,n)$-$\Theta$-space.  The following are
equivalent.
\begin{enumerate}
\item The object $X$ is
$\mathscr{T}^\Gpd_n$-fibrant.
\item For all $0\leq k<n$, the maps $\bar{X}(i_k)\colon
  \bar{X}(O_k)\ra \bar{X}(O_{k+1})$ are weak 
  equivalences of spaces.
\item For all $\theta\in \ob \Theta_n$, the map $Xp\colon X[0]\ra
  X\theta$ induced by $p\colon \theta\ra [0]$ is a weak equivalence of
  spaces.
\end{enumerate}
\end{prop}
\begin{proof}
It is clear from the definition of $\mathscr{T}_n^\Gpd$ that (1) and
(2) are equivalent.  It is immediate that (3) implies (2); it remains
to show that (1) implies (3), which we will show by induction on $n$.
Note that there is nothing to prove if $n=0$.
Since $\mathscr{T}^\Gpd_n\supseteq 
\Gpd_{\Theta_{n-1}}$, we see that the object $T^*X$ is a
$(\infty,1)$-$\Theta$-groupoid, which is to say, a groupoid-like
complete Segal space, and thus all maps $(T^*X)\delta^{0}\colon
(T^*X)[0]\ra (T^*X)[m]$ are weak equivalences of spaces
\cite{rezk-ho-theory-of-ho-theory}*{Cor.\ 6.6}.  
Therefore, $X[0]\ra
X\theta$ is a weak equivalence for all $\theta = [m]([0],\dots,[0])$,
$m\geq0$.  Now consider the diagram
\[
\xymatrix{
{X[0]}  \ar[r]^-{\sim}  \ar[dr]_a
& {X[m]([0],\dots,[0])} \ar[r]^-{\sim} \ar[d]
& {X[1]([0])\times_{X[0]} \cdots \times_{X[0]} X[m]([0])} \ar[d]^b
\\
& {X[m](\theta_1,\dots,\theta_m)}  \ar[r]_-{\sim}
& {X[1](\theta_1)\times_{X[0]}\cdots\times_{X[0]}X[1](\theta_m)}
}\]
where $\theta_1,\dots,\theta_m\in \ob \Theta_{n-1}$.  To show that $a$
is a weak equivalence, it suffices to show that $b$ is, so it suffices
to show that $X[1](p)\colon X[1]([0])\ra X[1](\theta)$ is a weak
equivalence of spaces for all $\theta\in \Theta_{n-1}$.   Consider the diagram
\[
\xymatrix{
{\bar{X}(V[1](F[0]))} \ar[rr]^{\bar{X}(V[1](Fp))} \ar[rd]
&& {\bar{X}(V[1](F\theta))} \ar[ld]
\\
& {\bar{X}(V[1](\varnothing))}
}\]
The map $\bar{X}(V[1](Fp))$ is isomorphic to $X[1](p)\colon
X[1]([0])\ra X[1](\theta)$.  Let
$(x_0,x_1)$ be a point in $\bar{X}(\partial O_1)\approx
\bar{X}(V[1](\varnothing))$; the map induced on fibers over
$(x_0,x_1)$ by $\bar{X}(V[1](Fp))$ is isomorphic to 
\[
\umap_X(x_0,x_1)(p)\colon \umap_X(x_0,x_1)([0]) \ra \umap_X(x_0,x_1)(\theta).
\]
It is clear that $\umap_X(x_0,x_1)$ is a
$(\infty,n-1)$-$\Theta$-groupoid, and thus $\umap_X(x_0,x_1)(p)$ is a
weak equivalence of spaces by the inductive hypothesis.
\end{proof}

Let $c_\#\colon \Sp \rightleftarrows \sPsh(\Theta_n)$ denote the
adjoint pair where the left adjoint $c_\#$ sends a space $X$ to the
constant presheaf with value $X$.

\begin{prop}\label{prop:quillen-equiv-groupoids-spaces}
\forcepar
\begin{enumerate}
\item
The adjoint pair $c_\#\colon \Sp\rightleftarrows
\sPsh(\Theta_n)^\inj_{\mathscr{T}_{n,\infty}\cup
  \mathscr{T}_n^\Gpd}\noloc c^*$ 
is a Quillen equivalence.  
\item
For all $-2\leq k<\infty$ the adjoint pair
$c_\#\colon \Sp_{n+k}\rightleftarrows
\sPsh(\Theta_n)^\inj_{\mathscr{T}_{n,k}\cup \mathscr{T}_n^\Gpd}\noloc
c^*$ is a Quillen equivalence.
\end{enumerate}
\end{prop}
\begin{proof}
We first consider (1).  Observe that $c_\#$ preserves cofibrations,
and that caries all spaces to
$\mathscr{T}_{n,\infty}\cup \mathscr{T}^\Gpd_n$-fibrant objects by
\eqref{prop:quillen-equiv-groupoids-spaces}, and thus $c_\#$ preserves
weak equivalences.  Therefore the pair is a Quillen pair, and it is a
straightforward consequence of
\eqref{prop:quillen-equiv-groupoids-spaces}  that the natural map
$X\ra c^*c_\#X$ is always weak equivalence, the natural map
$c_\#c^*Y\ra Y$ is a weak equivalence for all
$\mathscr{T}_{n,\infty}\cup \mathscr{T}^\Gpd_n$-local objects $Y$, and
thus the pair is a Quillen equivalence.

The proof that we get a Quillen equivalence in (2)  proceeds in the
same way, once we observe that $c_\#$ carries $n+k$-truncated spaces to
$\mathscr{T}_{n,k}\cup \mathscr{T}^\Gpd_n$-local objects.
\end{proof}


\end{document}